\documentclass[11pt,fleqn]{article}
\usepackage{geometry}                
\usepackage{graphicx}
\usepackage{amsthm,amsmath}
\usepackage{amssymb,dsfont,mathrsfs}

\numberwithin{equation}{section}
\newtheorem{theorem}{Theorem}[section]

\newtheorem{proposition}{Proposition}[section]
\newtheorem{corollary}{Corollary}[section]
\newtheorem{remark}{Remark}[section]
\newtheorem{assumption}{Assumption}[section]

\def\ba{\boldsymbol{a}}

\def\bg{\boldsymbol{g}}

\def\bl{\boldsymbol{l}}

\def\bu{\boldsymbol{u}}
\def\bw{\boldsymbol{w}}
\def\bx{\boldsymbol{x}}
\def\by{\boldsymbol{y}}

\def\bL{\boldsymbol{L}}

\def\bY{\boldsymbol{Y}}

\def\balpha{\boldsymbol{\alpha}}
\def\bbeta{\boldsymbol{\beta}}

\def\bphi{\boldsymbol{\phi}}
\def\bpi{\boldsymbol{\pi}}

\def\bzero{\mathbf{0}}
\def\bone{\mathbf{1}}
\def\calA{{\cal A}}

\def\calS{{\cal S}}

\def\calV{{\cal V}}

\def\calL{{\cal L}}

\def\scrI{\mathscr{I}}

\title{Stability condition of a two-dimensional QBD process and its application to estimation of efficiency for two-queue models} 
\author{Toshihisa Ozawa \\ 
Faculty of Business Administration, Komazawa University \\
1-23-1 Komazawa, Setagaya-ku, Tokyo 154-8525, Japan \\
E-mail: toshi@komazawa-u.ac.jp}

\begin{document}

\maketitle

\begin{abstract}
In order to analyze stability of a two-queue model, we consider a two-dimensional quasi-birth-and-death process (2d-QBD process), denoted by $\{\bY(t)\}=\{((L_1(t),L_2(t)),J(t))\}$. The two-dimensional process $\{(L_1(t),L_2(t))\}$ on $\mathbb{Z}_+^2$ is called a level process, where the individual processes $\{L_1(t)\}$ and $\{L_2(t)\}$ are assumed to be skip free. The supplemental process $\{J(t)\}$ is called a phase process and it takes values in a finite set.
The 2d-QBD process is a CTMC, in which the transition rates of the level process vary according to the state of the phase process like an ordinary (one-dimensional) QBD process. 
%
In this paper, we first state the conditions ensuring a 2d-QBD process is positive recurrent or transient and then demonstrate that the efficiency of a two-queue model can be estimated by using the conditions we obtained. 

\smallskip
{\it Key wards}: continuous-time Markov chain, quasi-birth-and-death process, stability, positive recurrence, Foster's criterion, matrix analytic method

\end{abstract}

%
%
\section{Introduction} \label{sec:intro}

Consider a two-queue model that consists of two customer classes corresponding to two queues and several servers serving customers according to some kind of service policy. In contrast to single-queue models, the stability condition of such a two-queue model is often non-trivial.  
For example, consider an $M/M/1$ queue with setup times (model 1) and an $M_1,M_2/M_1,M_2/1$ nonpreemptive-priority queue with setup times (model 2). Let $\lambda_1$ and $\mu_1$ be the arrival and service rates of model 1, respectively, and $\lambda_{2,1}$, $\lambda_{2,2}$, $\mu_{2,1}$ and $\mu_{2,2}$ those of model 2.  The traffic intensities of the models are given by $\rho_1=\lambda_1/\mu_1$ and $\rho_2=\lambda_{2,1}/\mu_{2,1}+\lambda_{2,2}/\mu_{2,2}$, respectively. 
It is well known that model 1 is stable if $\rho_1<1$, where we say a queueing model is stable if the continuous-time Markov chain (CTMC) representing the behavior of the queueing model is positive recurrent. On the other hand, model 2 may not be stable even if $\rho_2<1$, and this means that the stability condition of model 2 cannot be given only by using the traffic intensity. A reason why this phenomenon occurs in model 2 is that the queue of high-priority customers sometimes becomes empty even if the system is overloaded and at that time the server needs setup time to start service for a low-priority customer. 
With respect to model 2, letting the value of $\lambda_{2,1}$ or $\lambda_{2,2}$ vary with fixing the other parameters, we can see that there exists $\rho^*<1$ such that model 2 is stable if $\rho_2<\rho^*$ and it is unstable (i.e., the corresponding CTMC is transient) if $\rho_2> \rho^*$. %
The value of $\rho^*$ depends on the model parameters (see Section \ref{sec:example}). 
This $\rho^*$ is a measure to estimate efficiency of model 2 since it corresponds to the maximum ratio of server ability devoted to customer service under a certain condition. We, therefore, call $\rho^*$ the efficiency of the model. $\rho^*$ is also a measure to estimate efficiency of the service policy. Note that the value of $(\lambda_{2,1},\lambda_{2,2})$ when $\rho_2$ equals $\rho^*$ corresponds to the maximum throughput vector of customers. 

Another typical example is an N-model \cite{Tezcan13}, which consists of two customer classes and two server pools. Let $m_1$ be the number of servers in the first server pool and $m_2$ that in the second server pool. While the servers in the first server pool can serve only class-1 customers, those in the second server pool can serve customers of both classes, where class-1 customers have priority over class-2 customers. 
Assume Poisson arrivals and exponential services. Let $\lambda_1$ and $\lambda_2$ be the arrival rates of class-1 customers and class-2 customers, respectively, and let $\mu_1$ and $\mu_2$ be the service rates of class-1 customers and class-2 customers. The traffic intensity of the N-model per server is given by $\rho=(\lambda_1/\mu_1+\lambda_2/\mu_2)/(m_1+m_2)$. 
Like the first example, the condition ``$\rho<1$" does not ensure the N-model is stable. When the value of $\lambda_1$ or $\lambda_2$  varies with fixing the other parameters, there exists $\rho^*<1$ such that the N-model is stable if $\rho<\rho^*$ and it is unstable if $\rho>\rho^*$. This $\rho^*$ is a measure to estimate (total) efficiency of the N-model. 
Stability of N-models has been analyzed in \cite{Tezcan13}. 

In order to evaluate the efficiency of a two-queue model, it suffices to know the stability condition of the two-queue model. We, therefore, consider a two-dimensional quasi-birth-and-death process (2d-QBD process, for a discrete-time version of 2d-QBD process, see \cite{Ozawa13}) as a stochastic model representing the behavior of the two-queue model and obtain the stability condition of the 2d-QBD process. 
Denote a 2d-QBD process by $\{\bY(t)\}=\{((L_1(t),L_2(t)),J(t))\}$. The two-dimensional process $\{(L_1(t),L_2(t))\}$ on $\mathbb{Z}_+^2$ is called a level process, where the individual processes $\{L_1(t)\}$ and $\{L_2(t)\}$ are assumed to be skip free. The supplemental process $\{J(t)\}$ is called a phase process and it takes values in a finite set.
The 2d-QBD process is a CTMC, in which the transition rates of the level process vary according to the state of the phase process like an ordinary QBD process. This modulation is space homogeneous except for the boundaries of $\mathbb{Z}_+^2$. 
In the same way as ordinary QBD processes \cite{Latouche99,Neuts94}, stochastic models arising from various two-queue models and two-node queueing networks with Markovian arrival processes (MAPs) and phase-type services can be represented as 2d-QBD processes. 
Furthermore, two-queue models with various service policies such as nonpreemptive priority, $K$-limited service, server vacation and server  setup  can also be represented as 2d-QBD processes (for the case of ordinary QBD process, see \cite{Ozawa04}). 
Our first aim is to explicitly state the conditions ensuring a 2d-QBD process is positive recurrent or transient as a main theorem, and the second one is to demonstrate that the efficiency of a two-queue model can be evaluated by using the conditions we obtained. 
Here, it should be emphasized that we do not intend to analyze specific queueing models. Instead, we present a general-purpose way to analyze stability of Markovian two-queue models as well as Markovian two-node queueing networks. 

In the proof of the main theorem, we use a discrete-time 2d-QBD process obtained from the original 2d-QBD process by uniformization. The discrete-time 2d-QBD process has the same stationary distribution as the original 2d-QBD process and the stability condition of the latter is coincident with that of the former. 
A discrete-time 2d-QBD process \textit{without a phase process} is a two-dimensional skip-free reflecting random walk (2d-RRW) and stability of 2d-RRWs has been studied in a lot of literature (see \cite{Fayolle95} and references therein). Especially, remarkable results have been obtained in \cite{Fayolle89,Malyshev81}. In this paper, following their results, we analyze stability of the discrete-time 2d-QBD process obtained from the original QBD process. 
Key notions we use are ``induced Markov chain" and ``mean increment vector" \cite{Fayolle95,Malyshev81}. An induced Markov chain is a subprocess generated from a 2d-RRW and the mean increment vector with respect to the induced Markov chain is the vector of the expected increments of the 2d-RRW evaluated by using the stationary distribution of the induced Markov chain. 
%
%
The notion of induced Markov chain can be applied to discrete-time 2d-QBD processes as well as continuous-time 2d-QBD processes. 
Since a (continuous-time) 2d-QBD process is a CTMC, we define the mean transition rate vectors for the 2d-QBD process, instead of the mean increment vectors.
%
%
The obtained conditions ensuring a discrete-time 2d-QBD processes (resp.\ continuous-time 2d-QBD process) is positive recurrent or transient are represented in terms of the mean increment vectors (resp.\ mean transition rate vectors). We prove our results by using a kind of Foster's criterion. 
%

Here, we briefly comment on fluid limits and fluid models. As is well-known, one of the most useful methods for analyzing stability of queueing networks including two-queue models is the combination of fluid limits and fluid models (see \cite{Bramson08} and references therein). Applying that method to a queueing model with a service policy, we must give fluid equations to represent the service policy and prove the fluid model corresponding to the original queueing model is stable or not. It is not always an easy task.  
On the other hand, in our method, once a queueing model is represented as a 2d-QBD process, we can see the queueing model is stable or not, by using the conditions we obtained. 

The rest of the paper is organized as follows.  
In Section \ref{sec:model}, the 2d-QBD process is described in detail and the conditions ensuring it is positive recurrent or transient are stated as a main theorem. 
The main theorem is proved in Section \ref{sec:proof}. 
In Section \ref{sec:example}, via two simple examples, we demonstrate that the efficiency of a two-queue model can be evaluated by using our results.
The paper concludes with some remarks in Section \ref{sec:concluding}. 

\bigskip
\textit{Notations.} 
$\mathbb{Z}$ is the set of all integers, $\mathbb{Z}_+$ that of all nonnegative integers and $\mathbb{N}$ that of all positive integers. Define $\mathbb{H}$, $\mathbb{H}_+$ and $\mathbb{H}_-$ as $\mathbb{H}=\{-1,0,1\}$, $\mathbb{H}_+=\{0,1\}$ and $\mathbb{H}_-=\{-1,0\}$, respectively. 
$O$ is a matrix of $0$'s, $\bone$ is a column vector of $1$'s and $\bzero$ is a column vector of $0$'s. Their dimensions are determined in context, but if the dimensions should be specified, we denote them by subscripts. For example, $\bone_k$ is a $k\times 1$ vector of $1$'s. 
$I$ is the identity matrix.

%
%
\section{Model description and stability condition} \label{sec:model}

\subsection{2d-QBD process and related CTMCs} \label{sec:2dQBD}

A 2d-QBP process $\{\bY(t)\}=\{((L_1(t),L_2(t)),J(t))\}$ is a CTMC on a state space $\calS$ given as 
\[
\calS = (\{0\}\times\{0\}\times S_0) \cup (\mathbb{N}\times\{0\}\times S_1) \cup (\{0\}\times\mathbb{N}\times S_2) \cup (\mathbb{N}\times\mathbb{N}\times S_+), 
\]
where for $i\in\{0,1,2,+\}$, $S_i=\{1,2,\ldots,s_i\}$ and $s_i$ is the cardinality of $S_i$.
The infinitesimal generator of $\{\bY(t)\}$, $Q$, is represented in block form as 
\[
Q=\left( Q_{(l_1,l_2),(l_1',l_2')}; (l_1,l_2),(l_1',l_2')\in\mathbb{Z}_+^2 \right),
\]
where each block $Q_{(l_1,l_2),(l_1',l_2')}$ is given as, for some $i$ in $\{0,1,2,+\}$, 
\[
Q_{(l_1,l_2),(l_1',l_2')}=\left( q_{((l_1,l_2),j),((l_1',l_2'),j')}; j,j'\in S_i \right)
\]
and for $((l_1,l_2),j)\ne ((l_1',l_2'),j')$, $q_{((l_1,l_2),j),((l_1',l_2'),j')}$ is the transition rate from $((l_1,l_2),j)$ to $((l_1',l_2'),j')$. 
Since $\{L_1(t)\}$ and $\{L_2(t)\}$ are skip free, the block matrices can be given in terms of 36 matrices 
$
A^{(0)}_{k_1,k_2},\ A^{(1)}_{k_1,k_2},\ A^{(2)}_{k_1,k_2},\ A^{(+)}_{k_1,k_2},\ k_1,k_2\in\mathbb{H}, 
$
as follows: for $(l_1,l_2),(l_1',l_2')\in\mathbb{Z}_+^2$, 
\begin{align}
&Q_{(l_1,l_2),(l_1',l_2')} = \left\{ \begin{array}{ll} 
A^{(0)}_{\varDelta l_1,\varDelta l_2}, 
 & \mbox{if $(l_1,l_2)=(0,0)$, $\varDelta l_1,\varDelta l_2\in\mathbb{H}_+$}, \cr
 & \mbox{if  $(l_1',l_2')=(0,0)$, $\varDelta l_1,\varDelta l_2\in\mathbb{H}_-$}, \cr
 & \mbox{if  $(l_1,l_2)=(1,0)$, $\varDelta l_1=-1$, $\varDelta l_2=1$}, \cr
 & \mbox{or if  $(l_1,l_2)=(0,1)$, $\varDelta l_1=1$, $\varDelta l_2=-1$}, \cr
A^{(1)}_{\varDelta l_1,\varDelta l_2}, 
& \mbox{if $l_1\ge 1$, $l_2=0$, $\varDelta l_1\in\mathbb{H}_+$, $\varDelta l_2\in\mathbb{H}_+$}, \cr
& \mbox{if $l_1\ge 2$, $l_2=0$, $\varDelta l_1=-1$, $\varDelta l_2\in\mathbb{H}_+$}, \cr
& \mbox{if $l_1\ge 1$, $l_2=1$, $\varDelta l_1\in\mathbb{H}_+$, $\varDelta l_2=-1$}, \cr
& \mbox{or if  $l_1\ge 2$, $l_2=1$, $\varDelta l_1=\varDelta l_2=-1$}, \cr
A^{(2)}_{\varDelta l_1,\varDelta l_2}, 
& \mbox{if $l_1=0$, $l_2\ge 1$, $\varDelta l_1\in\mathbb{H}_+$, $\varDelta l_2\in\mathbb{H}_+$}, \cr
& \mbox{if  $l_1=0$, $l_2\ge 2$, $\varDelta l_1\in\mathbb{H}_+$, $\varDelta l_2=-1$}, \cr
& \mbox{if  $l_1=1$, $l_2\ge 1$, $\varDelta l_1=-1$, $\varDelta l_2\in\mathbb{H}_+$}, \cr
& \mbox{or if  $l_1=1$, $l_2\ge 2$, $\varDelta l_1=\varDelta l_2=-1$}, \cr
A^{(+)}_{\varDelta l_1,\varDelta l_2}, 
& \mbox{if  $l_1\ge 1$, $l_2\ge 1$, $\varDelta l_1,\varDelta l_2\in\mathbb{H}_+$}, \cr 
& \mbox{if   $l_1\ge 2$, $l_2\ge 1$, $\varDelta l_1=-1$, $\varDelta l_2\in\mathbb{H}_+$}, \cr 
& \mbox{if   $l_1\ge 1$, $l_2\ge 2$, $\varDelta l_1\in\mathbb{H}_+$, $\varDelta l_2=-1$}, \cr 
& \mbox{or if   $l_1\ge 2$, $l_2\ge 2$, $\varDelta l_1=\varDelta l_2=-1$}, \cr 
O, & \mbox{otherwise}, \end{array} \right. 
\label{eq:Q_blocks}
\end{align}
where $\varDelta l_1=l_1'-l_1$ and $\varDelta l_2=l_2'-l_2$ (see Fig.\ \ref{fig:transition_proba}). The dimensions of the block matrices are determined in context; for example, the dimension of each $A^{(+)}_{k_1,k_2}$ is $s_+\times s_+$ and that of $A^{(1)}_{0,-1}$ is $s_+\times s_1$. 
We assume the following condition throughout the paper. 
\begin{assumption} \label{as:Yt_irreducible}
The CTMC $\{\bY(t)\}$ is irreducible.
\end{assumption}

Since $Q$ is an infinitesimal generator, for $i\in\{0,1,2,+\}$, the off-diagonal elements of $A^{(i)}_{0,0}$ are nonnegative and the diagonal elements are negative under Assumption \ref{as:Yt_irreducible}; for $(k_1,k_2)\ne (0,0)$, $A^{(i)}_{k_1,k_2}$ is nonnegative. 
For $k_1,k_2\in\mathbb{H}$, define matrices $A^{(+)}_{*,k_2}$, $A^{(+)}_{k_1,*}$ and $A^{(+)}_{*,*}$ as $A^{(+)}_{*,k_2} = \sum_{k_1'\in\mathbb{H}} A^{(+)}_{k_1',k_2}$, $A^{(+)}_{k_1,*} = \sum_{k_2'\in\mathbb{H}} A^{(+)}_{k_1,k_2'}$ and $A^{(+)}_{*,*} = \sum_{k_1',k_2'\in\mathbb{H}} A^{(+)}_{k_1',k_2'}$. 
Furthermore, for $k_1,k_2\in\mathbb{H}$, define $A^{(1)}_{*,k_2}$ and $A^{(2)}_{k_1,*}$ as $A^{(1)}_{*,k_2} = \sum_{k_1'\in\mathbb{H}} A^{(1)}_{k_1',k_2}$ and $A^{(2)}_{k_1,*} = \sum_{k_2'\in\mathbb{H}} A^{(2)}_{k_1,k_2'}$.
%
%
%

\begin{figure}[htbp]
\begin{center}
\includegraphics[width=10cm,trim=150 280 150 40]{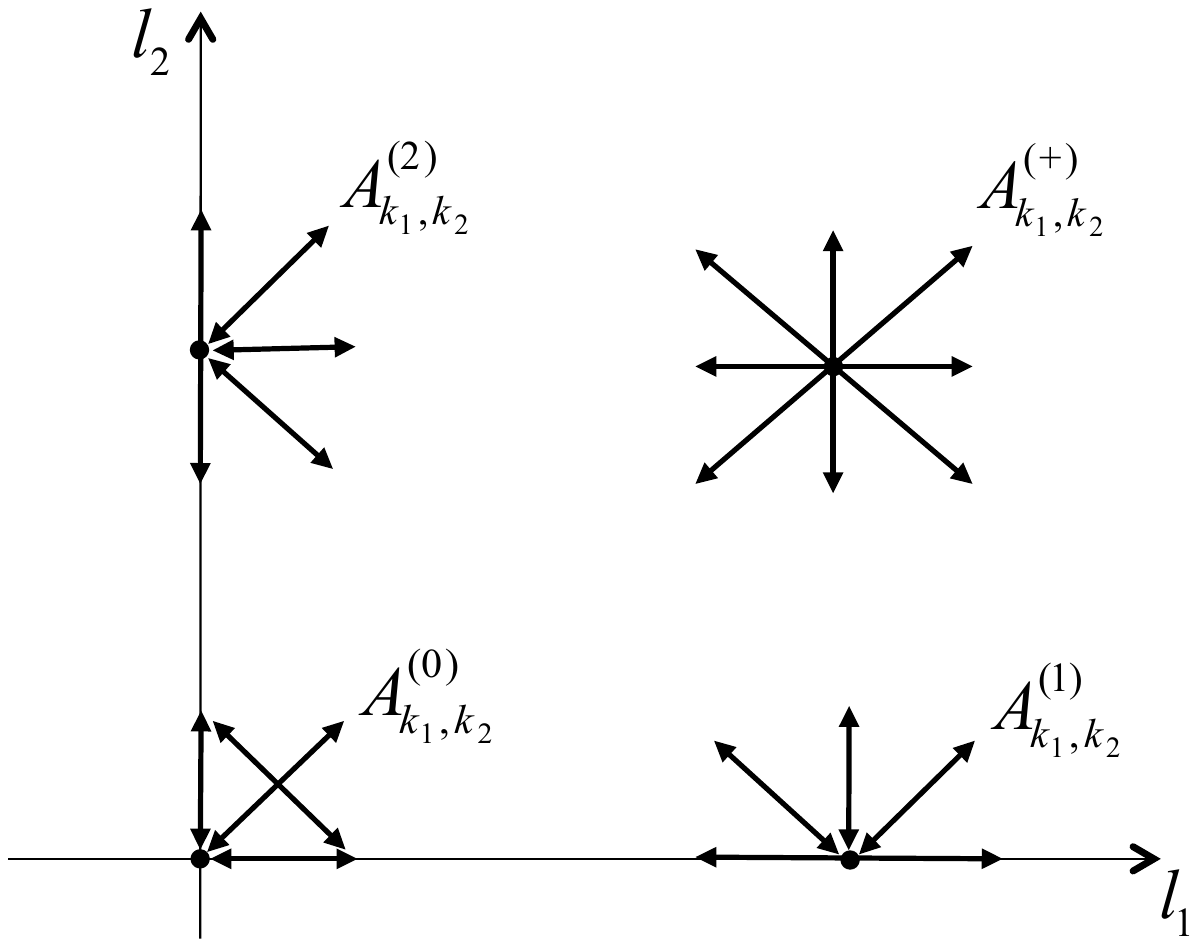} 
\caption{Transition rates of the 2d-QBD process}
\label{fig:transition_proba}
\end{center}
\end{figure}

%
Next, we define three kinds of CTMC generated from $\{\bY(t)\}$ by removing one or two boundaries. Denote them by $\{\bY^{(+)}(t)\}=\{((L^{(+)}_1(t),L^{(+)}_2(t)),J^{(+)}(t))\}$, $\{\bY^{(1)}(t)\}=\{((L^{(1)}_1(t),L^{(1)}_2(t)),J^{(1)}(t))\}$ and $\{\bY^{(2)}(t)\}=\{((L^{(2)}_1(t),L^{(2)}_2(t)),J^{(2)}(t))\}$, respectively.  
$\{\bY^{(+)}(t)\}$ is a CTMC on the state space $\calS^{(+)}=\mathbb{Z}^2\times S_+$ and it is generated from $\{\bY(t)\}$ by removing the boundaries on the $l_1$ and $l_2$-axes.  
The infinitesimal generator of $\{\bY^{(+)}(t)\}$, $Q^{(+)}$, is represented in block form as
\[
Q^{(+)}=\left( Q^{(+)}_{(l_1,l_2),(l_1',l_2')}; (l_1,l_2),(l_1',l_2')\in\mathbb{Z}^2 \right),
\]
and each block $Q^{(+)}_{(l_1,l_2),(l_1',l_2')}$ is given as 
\begin{align}
&Q^{(+)}_{(l_1,l_2),(l_1',l_2')}
= \left\{ \begin{array}{ll} A^{(+)}_{\varDelta l_1,\varDelta l_2}, & \mbox{if $\varDelta l_1,\varDelta l_2\in\mathbb{H}$}, \cr 
O, & \mbox{otherwise}, \end{array} \right. 
\end{align}
where $\varDelta l_1=l_1'-l_1$ and $\varDelta l_2=l_2'-l_2$. 
%
%
$\{\bY^{(1)}(t)\}$ is a CTMC on the state space $\calS^{(1)} = (\mathbb{Z}\times\{0\}\times S_1) \cup (\mathbb{Z}\times\mathbb{N}\times S_+)$ and it is generated from $\{\bY(t)\}$ by removing the boundary on the $l_2$-axis.  
The infinitesimal generator of $\{\bY^{(1)}(t)\}$, $Q^{(1)}$, is represented in block form as
\[
Q^{(1)} = \left( Q^{(1)}_{(l_1,l_2),(l_1',l_2')}; (l_1,l_2),(l_1',l_2')\in\mathbb{Z}\times\mathbb{Z}_+ \right),
\]
and each block $Q^{(1)}_{(l_1,l_2),(l_1',l_2')}$ is given as 
\begin{align}
&Q^{(1)}_{(l_1,l_2),(l_1',l_2')} 
= \left\{ \begin{array}{ll} 
A^{(1)}_{\varDelta l_1,\varDelta l_2}, 
& \mbox{if $l_2=0$, $\varDelta l_1\in\mathbb{H}$, $\varDelta l_2\in\mathbb{H}_+$}, \cr
& \mbox{or if $l_2=1$, $\varDelta l_1\in\mathbb{H}$, $\varDelta l_2=-1$}, \cr
A^{(+)}_{\varDelta l_1,\varDelta l_2}, 
& \mbox{if  $l_2\ge 1$, $\varDelta l_1\in\mathbb{H}$, $\varDelta l_2\in\mathbb{H}_+$}, \cr 
& \mbox{or if $l_2\ge 2$, $\varDelta l_1\in\mathbb{H}$, $\varDelta l_2=-1$}, \cr
O, & \mbox{otherwise}, 
\end{array} \right. 
\end{align}
where $\varDelta l_1=l_1'-l_1$ and $\varDelta l_2=l_2'-l_2$. 
%
$\{\bY^{(2)}(t)\}$ is a CTMC on the state space $\calS^{(2)} = (\{0\}\times\mathbb{Z}\times S_2) \cup (\mathbb{N}\times\mathbb{Z}\times S_+)$ and it is generated from $\{\bY(t)\}$ by removing the boundary on the $l_1$-axis. 
The infinitesimal generator of $\{\bY^{(2)}(t)\}$, $Q^{(2)}$, is represented in block form as
\[
Q^{(2)} = \left( Q^{(2)}_{(l_1,l_2),(l_1',l_2')}; (l_1,l_2),(l_1',l_2')\in\mathbb{Z}_+\times\mathbb{Z} \right),
\]
and each block $Q^{(2)}_{(l_1,l_2),(l_1',l_2')}$ is given as 
\begin{align}
&Q^{(2)}_{(l_1,l_2),(l_1',l_2')} 
= \left\{ \begin{array}{ll} 
A^{(2)}_{\varDelta l_1,\varDelta l_2}, 
& \mbox{if $l_1=0$, $\varDelta l_1\in\mathbb{H}_+$, $\varDelta l_2\in\mathbb{H}$}, \cr
& \mbox{or if $l_1=1$, $\varDelta l_1=-1$, $\varDelta l_2\in\mathbb{H}$}, \cr
A^{(+)}_{\varDelta l_1,\varDelta l_2}, 
& \mbox{if  $l_1\ge 1$, $\varDelta l_1\in\mathbb{H}_+$, $\varDelta l_2\in\mathbb{H}$}, \cr 
& \mbox{or if $l_1\ge 2$, $\varDelta l_1=-1$, $\varDelta l_2\in\mathbb{H}$}, \cr
O, & \mbox{otherwise}, 
\end{array} \right. 
\end{align}
where $\varDelta l_1=l_1'-l_1$ and $\varDelta l_2=l_2'-l_2$. 
The CTMCs $\{\bY^{(+)}(t)\}$, $\{\bY^{(1)}(t)\}$ and $\{\bY^{(2)}(t)\}$ are used for defining the induced CTMCs of $\{\bY(t)\}$.

%
%
\subsection{Induced CTMCs and mean transition rate vectors} \label{sec:mean_increment}

We define the induced CTMCs and mean transition rate vectors of the 2d-QBD process $\{\bY(t)\}$, according to \cite{Fayolle95,Malyshev81}. 
%
For $\{\bY(t)\}$, there are three induced CTMCs: $\calL^{(+)}$, $\calL^{(1)}$ and $\calL^{(2)}$. The induced CTMC $\calL^{(+)}$ is the phase process of $\{\bY^{(+)}(t)\}$, i.e., $\calL^{(+)} = \{J^{(+)}(t)\}$. The state space of $\calL^{(+)}$ is given by $S_+$ and the infinitesimal generator by $A^{(+)}_{*,*}$. 
The induced CTMCs $\calL^{(1)}$ and $\calL^{(2)}$ are the non-space-homogeneous parts of $\{\bY^{(1)}(t)\}$ and $\{\bY^{(2)}(t)\}$, respectively, and they are given as $\calL^{(1)} = \{(L^{(1)}_2(t),J^{(1)}(t))\}$ and $\calL^{(2)} = \{(L^{(2)}_1(t),J^{(2)}(t))\}$. 
The state space of $\calL^{(1)}$ is given by $(\{0\}\times S_1)\cup(\mathbb{N}\times S_+)$ and that of $\calL^{(2)}$ by $(\{0\}\times S_2)\cup(\mathbb{N}\times S_+)$. $\calL^{(1)}$ and $\calL^{(2)}$ are ordinary QBD processes and  their infinitesimal generators, denoted by $A^{(1)}_*$ and $A^{(2)}_*$, are given in block tri-diagonal form as 
\[
A^{(1)}_* = 
\begin{pmatrix}
A^{(1)}_{*,0} & A^{(1)}_{*,1} & & & \cr
A^{(1)}_{*,-1} & A^{(+)}_{*,0} & A^{(+)}_{*,1} & & \cr
& A^{(+)}_{*,-1} & A^{(+)}_{*,0} & A^{(+)}_{*,1} & \cr
& & \ddots & \ddots & \ddots 
\end{pmatrix},\ 
A^{(2)}_* = 
\begin{pmatrix}
A^{(2)}_{0,*} & A^{(2)}_{1,*} & & & \cr
A^{(2)}_{-1,*} & A^{(+)}_{0,*} & A^{(+)}_{1,*} & & \cr
& A^{(+)}_{-1,*} & A^{(+)}_{0,*} & A^{(+)}_{1,*} & \cr
& & \ddots & \ddots & \ddots 
\end{pmatrix}. 
\]

In the 2d-QBD process arising from a two-queue model, the induced CTMCs $\calL^{(+)}$, $\calL^{(1)}$ and $\calL^{(2)}$ may become reducible. For example, in the 2d-QBD process arising from a two-class non-preemptive priority queue with setup times, which will be considered in Section \ref{sec:example}, $\calL^{(+)}$ is reducible and has just one irreducible class (closed communication class). 
Furthermore, in that 2d-QBD process, $\calL^{(1)}$ is reducible and has no irreducible class.  
Therefore, we assume the following conditions throughout the paper.

\begin{assumption} \label{as:calL_irreducible}
The induced CTMC $\calL^{(+)}$ has just one irreducible class. 
\end{assumption}

\begin{assumption} \label{as:calL12_irreducible}
Both the induced CTMCs $\calL^{(1)}$ and $\calL^{(2)}$ have at most one irreducible class. If $\calL^{(1)}$ (resp.\ $\calL^{(2)}$) has just one irreducible class, the irreducible class is a countably infinite set and every state in the irreducible class is accessible from any state of $\calL^{(1)}$ (resp.\ $\calL^{(2)}$). 
\end{assumption}

These assumptions are not essential and we can easily extend them. For example, $\calL^{(+)}$ may have several irreducible classes. We adopt these assumptions since they are sufficiently wide in analyzing queueing models and they also make discussion of stability for 2d-QBD processes simple. 
%
Under Assumption \ref{as:calL_irreducible}, since the state space of $\calL^{(+)}$, $S_+$, is finite, $\calL^{(+)}$ always has a unique stationary distribution. We denote it by $\bpi^{(+)}_{*,*}$.
The mean transition rate vector with respect to $\calL^{(+)}$, $\ba^{(+)}=(a^{(+)}_1,a^{(+)}_2)$, is defined as 
\begin{align}
&a^{(+)}_1 = \bpi^{(+)}_{*,*} (-A^{(+)}_{-1,*}+A^{(+)}_{1,*}) \bone,\quad 
a^{(+)}_2 = \bpi^{(+)}_{*,*} (-A^{(+)}_{*,-1}+A^{(+)}_{*,1}) \bone. 
\end{align}
From the definition, we see that $\ba^{(+)}$ is the mean transition rate vector of the level process  $\{(L^{(+)}_1(t),L^{(+)}_2(t))\}$ of $\{\bY^{(+)}(t)\}$. 
%

If the induced CTMC $\calL^{(1)}$ (resp.\ $\calL^{(2)}$) has no irreducible classes, all the states of $\calL^{(1)}$ (resp.\ $\calL^{(2)}$) are transient. In that case, we do not define the mean transition rate vector with respect to $\calL^{(1)}$ (resp.\ $\calL^{(2)}$). 
If $\calL^{(1)}$ (resp.\ $\calL^{(2)}$) has just one irreducible class, we denote the irreducible class by $\calS^{(1)}_{irr}$ (resp.\ $\calS^{(2)}_{irr}$).

\begin{remark} \label{re:calL12_properties}
Under Assumption \ref{as:calL12_irreducible}, if $\calL^{(1)}$ has just one irreducible class, $\calS^{(1)}_{irr}$ is countably infinite and we have that $\calS^{(1)}_{irr}\cap(\{0\}\times S_1)\ne\emptyset$ and, for every $k\in\mathbb{N}$, $\calS^{(1)}_{irr}\cap(\{k\}\times S_+)\ne\emptyset$. 
%
%
A similar result also holds for $\calL^{(2)}$. 
\end{remark}

We see from Remark \ref{re:calL12_properties} that, under Assumption \ref{as:calL12_irreducible}, if $\calL^{(1)}$ (resp.\ $\calL^{(2)}$) has just one irreducible class, it is an ordinary QBD process. 
In that case, let $R^{(1)}$ (resp.\ $R^{(2)}$) be the rate matrix of $\calL^{(1)}$ (resp.\ $\calL^{(2)}$).  $R^{(1)}$ and $R^{(2)}$ are the minimum nonnegative solutions to the following matrix quadratic equations:
\begin{align}
&(R^{(1)})^2 A^{(+)}_{*,-1} + R^{(1)} A^{(+)}_{*,0} + A^{(+)}_{*,1} = O,\\ 
&(R^{(2)})^2 A^{(+)}_{-1,*} + R^{(2)} A^{(+)}_{0,*} + A^{(+)}_{1,*} = O.
\end{align}
If $a^{(+)}_2<0$, then $\calL^{(1)}$ has a unique stationary distribution $\bpi^{(1)}_*=(\bpi^{(1)}_{*,l},l\in\mathbb{Z}_+)$ given as 
\begin{align}
\bpi^{(1)}_{*,l} = \bpi^{(1)}_{*,0} A^{(1)}_{*,1}  (-A^{(+)}_{*,0}-R^{(1)} A^{(+)}_{*,-1})^{-1} (R^{(1)})^{l-1},\ l\ge 1,
\end{align}
and the mean transition rate vector with respect to $\calL^{(1)}$, $\ba^{(1)}=(a^{(1)}_1,a^{(1)}_2)$, is defined as
\begin{align}
a^{(1)}_1 &= \bpi_{*,0}^{(1)}\big( (-A_{-1,0}^{(1)}+A_{1,0}^{(1)})\bone_{s_1}+(-A_{-1,1}^{(1)}+A_{1,1}^{(1)})\bone_{s_+} \big) \cr
&\quad + \bpi_{*,1}^{(1)}\big( (-A_{-1,-1}^{(1)}+A_{1,-1}^{(1)})\bone_{s_1}+(-A^{(+)}_{-1,0}-A^{(+)}_{-1,1}+A^{(+)}_{1,0}+A^{(+)}_{1,1})\bone_{s_+} \big)\cr 
&\quad + \bpi_{*,2}^{(1)} (I-R^{(1)})^{-1} (-A^{(+)}_{-1,*}+A^{(+)}_{1,*})\bone_{s_+},\\
a^{(1)}_2 &= \bpi_{*,0}^{(1)} A_{*,1}^{(1)}\bone_{s_+} + \bpi_{*,1}^{(1)}\big( -A_{*,-1}^{(1)}\bone_{s_1}+A^{(+)}_{*,1}\bone_{s_+} \big) + \bpi_{*,2}^{(1)} (I-R^{(1)})^{-1} (-A^{(+)}_{*,-1}+A^{(+)}_{*,1})\bone_{s_+}  \cr
&= 0. 
\end{align}
If $a^{(+)}_2=0$, then $\calL^{(1)}$ is null recurrent and if $a^{(+)}_2>0$, then it is transient. In these cases, the mean transition rate vector $\ba^{(1)}$ is undefined. 
If $a^{(+)}_1<0$, then $\calL^{(2)}$ has a unique stationary distribution $\bpi^{(2)}_*=(\bpi^{(2)}_{*,l},l\in\mathbb{Z}_+)$ given as 
\begin{align}
\bpi^{(2)}_{*,l} = \bpi^{(2)}_{*,0} A^{(2)}_{1,*}  (-A^{(+)}_{0,*}-R^{(2)} A^{(+)}_{-1,*})^{-1} (R^{(2)})^{l-1},\ l\ge 1,
\end{align}
and the mean transition rate vector with respect to $\calL^{(2)}$, $\ba^{(2)}=(a^{(2)}_1,a^{(2)}_2)$, is defined as
\begin{align}
a^{(2)}_1 &= \bpi_{*,0}^{(2)} A_{1,*}^{(2)}\bone_{s_+} + \bpi_{*,1}^{(2)}\big( -A_{-1,*}^{(2)}\bone_{s_2}+A^{(+)}_{1,*}\bone_{s_+} \big) + \bpi_{*,2}^{(2)} (I-R^{(2)})^{-1} (-A^{(+)}_{-1,*}+A^{(+)}_{1,*})\bone_{s_+}  \nonumber \\
&= 0. \\
a^{(2)}_2 &= \bpi_{*,0}^{(2)}\big( (-A_{0,-1}^{(2)}+A_{0,1}^{(2)})\bone_{s_2}+(-A_{1,-1}^{(2)}+A_{1,1}^{(2)})\bone_{s_+} \big) \cr
&\quad + \bpi_{*,1}^{(2)}\big( (-A_{-1,-1}^{(2)}+A_{-1,1}^{(2)})\bone_{s_2}+(-A^{(+)}_{0,-1}-A^{(+)}_{1,-1}+A^{(+)}_{0,1}+A^{(+)}_{1,1})\bone_{s_+} \big)\cr 
&\quad + \bpi_{*,2}^{(2)} (I-R^{(2)})^{-1} (-A^{(+)}_{*,-1}+A^{(+)}_{*,1})\bone_{s_+}.
\end{align}
If $a^{(+)}_1=0$, then $\calL^{(2)}$ is null recurrent and if $a^{(+)}_1>0$, then $\calL^{(2)}$ is transient. In these cases, the mean transition rate vector $\ba^{(2)}$ is undefined. 
From the definitions, we see that, for $i\in\{1,2\}$, if $\ba^{(i)}$ is well defined, it is the mean transition rate vector of the level process $\{(L^{(i)}_1(t),L^{(i)}_2(t))\}$ of $\{\bY^{(i)}(t)\}$. 
%

%
%
\subsection{Positive recurrence and transience} \label{sec:stability_cond}

Conditions ensuring the 2d-QBD process is positive recurrent or transient are given as follows. We will prove this theorem in Section \ref{sec:proof}. 
%
\begin{theorem} \label{th:stability_cond1}
\begin{itemize}
\item[(i)] In the case where $a^{(+)}_1<0$ and $a^{(+)}_2<0$, the 2d-QBD process $\{\bY(t)\}$ is positive recurrent if $a^{(1)}_1<0$ and $a^{(2)}_2<0$, and it is transient if either $a^{(1)}_1>0$ or $a^{(2)}_2>0$. 
\item[(ii)] In the case where $a^{(+)}_1\ge 0$ and $a^{(+)}_2<0$, $\{\bY(t)\}$ is positive recurrent if $a^{(1)}_1<0$, and it is transient if $a^{(1)}_1>0$. 
\item[(iii)] In the case where $a^{(+)}_1<0$ and $a^{(+)}_2\ge 0$, $\{\bY(t)\}$ is positive recurrent if $a^{(2)}_2<0$, and it is transient if $a^{(2)}_2>0$. 
\item[(iv)] If one of $a^{(+)}_1$ and $a^{(+)}_2$ is positive and the other is non-negative, then $\{\bY(t)\}$ is transient.
\end{itemize}
\end{theorem}

\begin{remark} \label{re:further_study1}
The following cases are excluded from Theorem \ref{th:stability_cond1}. 

\medskip
\begin{tabular}{rl}
(a-1) & $a^{(+)}_1<0$, $a^{(+)}_2<0$, $a^{(1)}_1=0$ and $a^{(2)}_2\le 0$. \cr
(a-2) & $a^{(+)}_1<0$, $a^{(+)}_2<0$, $a^{(1)}_1\le 0$ and $a^{(2)}_2=0$. \cr
(b) & $a^{(+)}_1\ge 0$, $a^{(+)}_2<0$ and $a^{(1)}_1=0$. \cr
(c) & $a^{(+)}_1< 0$, $a^{(+)}_2\ge 0$ and $a^{(2)}_2=0$. \cr
(d) & $a^{(+)}_1=a^{(+)}_2=0$.
\end{tabular}
\medskip

We know that 2d-RRWs are null recurrent in the cases corresponding to  (a-1) through (c) (see Theorem 3.3.2 of \cite{Fayolle95}). 
Similar results are expected to hold for 2d-QBD processes. 
In \cite{Fayolle95}, the case corresponding to (d) is called the case of zero drifts. In that case, the 2d-QBD process may become positive recurrent (for the case of 2d-RRW, see Theorem 3.4.1 of \cite{Fayolle95}). 
To clarify these points, we need a method different from that used for proving Theorem \ref{th:stability_cond1} in Section \ref{sec:proof}.  We, therefore, leave it as a further study. 
\end{remark}

%
%
\section{Efficiency of two-queue models: examples} \label{sec:example} 

%
%
\subsection{Two-queue model}

We consider a queueing model with two customer classes, depicted in Fig.\ \ref{fig:twoqueue}. Class-1 customers arrive according to an arrival process with arrival rate $\lambda_1$ and enter queue 1 (Q$_1$). Class-2 customers arrive according to another arrival process with arrival rate $\lambda_2$ and enter queue 2 (Q$_2$). In the system, there are $c$ servers ($c\ge 1$) and they serve customers according to some kind of service policy. After completion of service, customers leave the system without reentrance. We refer to this queueing model as a two-queue model. 
Let $h_1$ and $h_2$ be the mean service times of class-1 customers and class-2 customers, respectively. Then, the traffic intensity of the two-queue model per server is given by $\rho=(\lambda_1 h_1+\lambda_2 h_2)/c$. 
We define the efficiency of the two-queue model, denoted by $\rho^*$, as follows. Let the value of $\lambda_1$ (or $\lambda_2$) increase without changing the stochastic nature of the arrival process up to the value at which the model becomes unstable for the first time. Denote that value of $\lambda_1$ (resp.\ $\lambda_2$) by $\lambda_1^*$ (resp.\ $\lambda_2^*$) and give $\rho^*$ as $\rho^*=(\lambda_1^* h_1+\lambda_2 h_2)/c$ (resp.\ $\rho^*=(\lambda_1 h_1+\lambda_2^* h_2)/c$). 
For example, if the original arrival process of class-1 customers is given by a Markovian arrival process (MAP) with representation $(C,D)$, then the arrival process of class-1 customers with arrival rate $\lambda_1$ is given by the MAP with representation $((\lambda_1/\hat{\lambda}_1) C, (\lambda_1/\hat{\lambda}_1) D)$, where $\hat{\lambda}_1$ is the mean arrival rate of the original MAP. 
We call $\rho^*$ the efficiency of the two-queue model. If it is possible to exhaustively use the ability of the servers for customer service, the value of $\rho^*$ becomes $1$.  The vector $(\lambda_1^*,\lambda_2)$ (resp.\ $(\lambda_1,\lambda_2^*)$) corresponds to the maximum throughput vector of the two-queue model.

\begin{figure}[htbp]
\begin{center}
\includegraphics[width=8cm,trim=250 420 300 30]{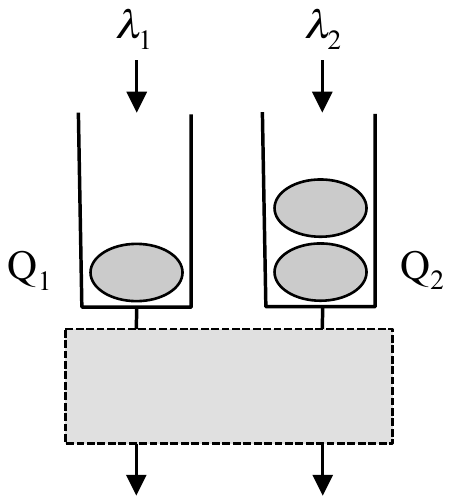} 
\caption{Two-queue model}
\label{fig:twoqueue}
\end{center}
\end{figure}

In the following subsections, we consider two kinds of priority queueing model in order to demonstrate how our results work. In each model, there are two queues that interact with each other and the stability condition of the model is not so trivial. 
Note that we do not intend to propose new queueing models here; we just present examples to understand our results. 

%
%
\subsection{Priority queue with setup times}

The first example is a single-server two-class non-preemptive priority queue with setup times. Class-1 customers arrive according to a Poisson process with intensity $\lambda_1$ and class-2 customers according to another Poisson process with intensity $\lambda_2$. Service times for class-1 customers are subject to an exponential distribution with mean $1/\mu_1$ and those for class-2 customers subject to another exponential distribution with mean $1/\mu_2$. The traffic intensity $\rho$ is given as $\rho=\lambda_1/\mu_1+\lambda_2/\mu_2$. 
Class-1 customers have non-preemptive priority over class-2 customers. 
The idle server needs a setup time to restart service for customers. Furthermore, after completing service for a class-1 customer (resp.\ class-2 customer), the server also needs a setup time if a customer to be served next is of class-2 (resp.\ of class-1). Setup times for class-1 customer's service are subject to an exponential distribution with mean $1/\gamma_1$ and those for class-2 customer's service subject to another exponential distribution with mean $1/\gamma_2$. 
We assume that the arrival processes, service times and setup times are mutually independent. 

For $i\in\{1,2\}$, let $L_i(t)$ be the number of class-$i$ customers in the system at time $t$. 
Let $J(t)$ be the server state at time $t$, which is defined as follows. When $L_1(t)=L_2(t)=0$, $J(t)$ takes the value of $1$, which means that the server is idle. 
When $L_1(t)>0$ and $L_2(t)=0$, $J(t)$ takes a value in $\{1,2\}$, where $J(t)=1$ means that the server is engaging in service for a class-1 customer and $J(t)=2$ that it is engaging in setup for class-1 customer's service.   
When $L_1(t)=0$ and $L_2(t)>0$, $J(t)$ takes a value in $\{1,2\}$, where $J(t)=1$ means that the server is engaging in service for a class-2 customer and $J(t)=2$ that it is engaging in setup for class-2 customer's service. 
When $L_1(t)>0$ and $L_2(t)>0$, $J(t)$ takes a value in $\{1,2,3,4\}$, where $J(t)=1$ means that the server is engaging in service for a class-1 customer, $J(t)=2$ that it is engaging in setup for class-1 customer's service, $J(t)=3$ that it is engaging in service for a class-2 customer and $J(t)=4$ that it is engaging in setup for class-2 customer's service. 
Then, the process $\{\bY(t)\}=\{((L_1(t),L_2(t)),J(t))\}$ is a 2d-QBD process and it is governed by the infinitesimal generator $Q$ composed of the following block matrices:
\begin{align*}
&A^{(+)}_{-1,0} = \begin{pmatrix}
\mu_1 & 0 & 0 & 0 \cr
0 & 0 & 0 & 0 \cr
0 & 0 & 0 & 0 \cr
0 & 0 & 0 & 0 
\end{pmatrix},\  
A^{(+)}_{0,0} = \begin{pmatrix}
-(\lambda+\mu_1) & 0 & 0 & 0 \cr
\gamma_1 & -(\lambda+\gamma_1) & 0 & 0 \cr
0 & 0 & -(\lambda+\mu_2) & 0 \cr
0 & 0 & \gamma_2 & -(\lambda+\gamma_2)
\end{pmatrix}, \\
&A^{(+)}_{0,-1} = \begin{pmatrix}
0 & 0 & 0 & 0 \cr
0 & 0 & 0 & 0 \cr
0 & \mu_2 & 0 & 0 \cr
0 & 0 & 0 & 0 
\end{pmatrix},\ 
A^{(1)}_{0,-1} = \begin{pmatrix}
0 & 0 \cr
0 & 0 \cr
0 & \mu_2 \cr
0 & 0  
\end{pmatrix},\ 
A^{(2)}_{-1,0} = \begin{pmatrix}
0 & \mu_1 \cr
0 & 0 \cr
0 & 0 \cr
0 & 0  
\end{pmatrix},\\
&A^{(+)}_{1,0}=\lambda_1 I,\ 
A^{(+)}_{0,1}=\lambda_2 I,\ 
A^{(+)}_{1,1}=A^{(+)}_{-1,1}=A^{(+)}_{1,-1}=A^{(+)}_{-1,-1}=O, 
\end{align*}
\begin{align*}
&A^{(1)}_{-1,0} = \begin{pmatrix}
\mu_1 & 0 \cr
0 & 0  
\end{pmatrix},\ 
A^{(1)}_{0,0} = \begin{pmatrix}
-(\lambda+\mu_1) & 0 \cr
\gamma_1 & -(\lambda+\gamma_1)  
\end{pmatrix},\ 
A^{(1)}_{0,1} = \lambda_2 \begin{pmatrix}
1 & 0 & 0 & 0 \cr
0 & 1 & 0 & 0  
\end{pmatrix},\\
&A^{(1)}_{1,0}=\lambda_1 I,\ 
A^{(1)}_{1,1}=A^{(1)}_{-1,1}=O,\ A^{(1)}_{1,-1}=A^{(1)}_{-1,-1}=O, 
\end{align*}
\begin{align*}
&A^{(2)}_{0,0} = \begin{pmatrix}
-(\lambda+\mu_2) & 0 \cr
\gamma_2 & -(\lambda+\gamma_2)  
\end{pmatrix},\ 
A^{(2)}_{0,-1} = \begin{pmatrix}
\mu_2 & 0 \cr
0 & 0  
\end{pmatrix},\ 
A^{(2)}_{1,0} = \lambda_1 \begin{pmatrix}
0 & 0 & 1 & 0 \cr
0 & 0 & 0 & 1  
\end{pmatrix},\\
&A^{(2)}_{0,1}=\lambda_2 I,\ 
A^{(2)}_{1,1}=A^{(2)}_{1,-1}=O,\ A^{(2)}_{-1,1}=A^{(2)}_{-1,-1}=O, 
\end{align*} 
\begin{align*}
&A^{(0)}_{-1,0} = \begin{pmatrix}
\mu_1 \cr 0  
\end{pmatrix},\ 
A^{(0)}_{0,0} = -\lambda,\ 
A^{(0)}_{0,-1} = \begin{pmatrix}
\mu_2 \cr 0  
\end{pmatrix},\ 
A^{(0)}_{1,0} = \lambda_1 \begin{pmatrix}
0 & 1  
\end{pmatrix},\ 
A^{(0)}_{0,1} = \lambda_2 \begin{pmatrix}
0 & 1  
\end{pmatrix},\\
&A^{(0)}_{1,1}=\bzero^\top,\ A^{(0)}_{-1,-1}=\bzero,\ 
A^{(0)}_{-1,1}=A^{(0)}_{1,-1}=O, 
\end{align*} 
where $\lambda=\lambda_1+\lambda_2$. The state space of $\{\bY(t)\}$ is given by $\calS=(\{0\}\times\{0\}\times S_0)\cup(\{0\}\times\mathbb{N}\times S_1)\cup(\mathbb{N}\times\{0\}\times S_2)\cup(\mathbb{N}^2\times S_+)$, where $S_0=\{1\}$, $S_1=S_2=\{1,2\}$, $S_+=\{1,2,3,4\}$. 
%

The infinitesimal generator of the induced CTMC $\calL^{(+)}=\{J^{(+)}(t)\}$ is given by
\begin{align*}
&A^{(+)}_{*,*} = \begin{pmatrix}
0 & 0 & 0 & 0 \cr
\gamma_1 & -\gamma_1 & 0 & 0 \cr
0 & \mu_2 & -\mu_2 & 0 \cr
0 & 0 & \gamma_2 & -\gamma_2 
\end{pmatrix}. 
\end{align*}
Hence, $\calL^{(+)}$ is reducible and has just one irreducible class, which is $\{1\}$. The stationary distribution of $\calL^{(+)}$ is given by $\bpi^{(+)}_{*,*}=\begin{pmatrix} 1 & 0 & 0 & 0 \end{pmatrix}$ and the mean transition rate vector $\ba^{(+)}=(a^{(+)}_1,a^{(+)}_2)$ is given as  $a^{(+)}_1=\lambda_1-\mu_1$ and $a^{(+)}_2=\lambda_2>0$. 
%
The nonzero block matrices of the infinitesimal generator of the induced CTMC $\calL^{(1)}=\{(L^{(1)}_2(t),J^{(1)}(t))\}$ are given by
\begin{align*}
&A^{(+)}_{*,-1}=A^{(+)}_{0,-1},\ 
A^{(+)}_{*,0} = \begin{pmatrix}
-\lambda_2 & 0 & 0 & 0 \cr
\gamma_1 & -(\lambda_2+\gamma_1) & 0 & 0 \cr
0 & 0 & -(\lambda_2+\mu_2) & 0 \cr
0 & 0 & \gamma_2 & -(\lambda_2+\gamma_2) 
\end{pmatrix}, \ 
A^{(+)}_{*,1}=\lambda_2 I,\\ 
&A^{(1)}_{*,-1}=A^{(1)}_{0,-1},\ 
A^{(1)}_{*,0} = \begin{pmatrix}
-\lambda_2 & 0 \cr
\gamma_1 & -(\lambda_2+\gamma_1) 
\end{pmatrix},\ 
A^{(1)}_{*,1}=A^{(1)}_{0,1}.
\end{align*}
From the structure of these block matrices, we see that $\calL^{(1)}$ is reducible and has no irreducible classes. 
On the other hand, the nonzero block matrices of the infinitesimal generator of the induced CTMC $\calL^{(2)}=\{(L^{(2)}_1(t),J^{(2)}(t))\}$ are given by
\begin{align*}
&A^{(+)}_{-1,*}=A^{(+)}_{-1,0},\ 
A^{(+)}_{0,*} = \begin{pmatrix}
-(\lambda_1+\mu_1) & 0 & 0 & 0 \cr
\gamma_1 & -(\lambda_1+\gamma_1) & 0 & 0 \cr
0 & \mu_2 & -(\lambda_1+\mu_2) & 0 \cr
0 & 0 & \gamma_2 & -(\lambda_1+\gamma_2) 
\end{pmatrix}, \ 
A^{(+)}_{1,*}=\lambda_1 I,\\ 
&A^{(2)}_{-1,*}=A^{(2)}_{-1,0},\ 
A^{(2)}_{0,*} = \begin{pmatrix}
-\lambda_1 & 0 \cr
\gamma_1 & -(\lambda_1+\gamma_1) 
\end{pmatrix},\ 
A^{(2)}_{1,*}=A^{(2)}_{1,0}.
\end{align*}
From the structure of these block matrices, we see that $\calL^{(2)}$ is irreducible. Hence, if $a^{(+)}_1<0$, $\calL^{(2)}$ is positive recurrent and the mean transition rate vector $\ba^{(2)}=(a^{(2)}_1,a^{(2)}_2)$ is well defined. 
By Theorem \ref{th:stability_cond1}, if $a^{(+)}_1<0$ and $a^{(2)}_2<0$, $\{\bY(t)\}$ is positive recurrent; if $a^{(+)}_1\ge 0$ or if $a^{(+)}_1<0$ and $a^{(2)}_2>0$, it is transient. 

Fixing the values of $\mu_1$, $\mu_2$, $\gamma_1$ and $\gamma_2$ and setting $\lambda_1$ at a value satisfying $a^{(+)}_1=\mu_1-\lambda_1<0$, we can evaluate the value of $\lambda_2$ that makes $a^{(2)}$ equal $0$ by using the bisection method. We denote by $\lambda_2^*$ that value of $\lambda_2$. The efficiency of the model is given by $\rho^*=\lambda_1/\mu_1+\lambda_2^*/\mu_2$ and the maximum throughput vector by $(\lambda_1,\lambda_2^*)$. 
In Table \ref{tab:table1}, we give numerical examples when $\mu_1=\mu_2=1$, $\gamma_1=\gamma_2=2$ and $a^{(+)}_1<0$. From the table, we can see how setup times influence congestion of the system depending on $\lambda_1$. In this case, the efficiency of the model becomes minimum when the value of $\lambda_1$ is around $0.4$. 
Similar evaluation is available even if the arrival processes are replaced with MAPs and the service time distributions as well as the setup time distributions are replaced with phase-type distributions (PH-distributions). We give the representation of the model in that case in Appendix \ref{sec:setup_MAPPH}. 
\begin{table}[htp]
\caption{The values of $\lambda_2$ that make $a^{(2)}$ equal $0$.}
\begin{center}
\begin{tabular}{c|ccccccccc}
$\lambda_1$ & 0.1 & 0.2 & 0.3 & 0.4 & 0.5 & 0.6 & 0.7 & 0.8 & 0.9 \cr \hline
$\lambda_2^*$ & 0.821 & 0.678 & 0.557 & 0.453 & 0.361 & 0.278 & 0.202 & 0.131 & 0.064 \cr
$\rho^*$ & 0.922 & 0.878 & 0.857 & 0.853 & 0.861 & 0.878 & 0.902 & 0.931 & 0.964 
\end{tabular}
\end{center}
\label{tab:table1}
\end{table}%

%
%
\subsection{Two-queue model with an additional server}

The second example is a model related to the N-model. It is composed of two $M/M/1$ queues and an additional server. We denote the two queues by Q$_1$ and Q$_2$, respectively. The additional server can serve customers in both the queues, and customers in Q$_1$ have non-preemptive priority over those in Q$_2$, with respect to use of the additional server. 
It means that, after completing a service, if there exists at least one waiting customer in Q$_1$, the additional server next serves a customer in Q$_1$; if there are no waiting customers in Q$_1$ and there exists at least one waiting customer in Q$_2$, it next serves a customer in Q$_2$; otherwise, it becomes idle. 
Denote by $\lambda_1$ and $\lambda_2$ the arrival rates of Q$_1$ and Q$_2$, respectively, and by $\mu_1$ and $\mu_2$ the service rates of them, respectively. The traffic intensity per server is given by $\rho=(\lambda_1/\mu_1+\lambda_2/\mu_2)/3$. 

\begin{table}[htp]
\caption{States of the servers.}
\begin{center}
\begin{tabular}{cc|cc|cc|cc}
\multicolumn{2}{c|}{$L_1(t)=L_2(t)=0$} & \multicolumn{2}{c|}{$L_1(t)>0$, $L_2(t)=0$} & \multicolumn{2}{c|}{$L_1(t)=0$, $L_2(t)>0$} & \multicolumn{2}{c}{$L_1(t)>0$, $L_2(t)>0$} \cr \hline
$J(t)$ & $(j_1,j_2,j_3)$ & $J(t)$ & $(j_1,j_2,j_3)$ & $J(t)$ & $(j_1,j_2,j_3)$ & $J(t)$ & $(j_1,j_2,j_3)$ \cr \hline
1 & (0,0,0) & 1 & (1,0,1) & 1 & (0,2,2) & 1 & (1,2,1) \cr
2 & (1,0,0) & 2 & (1,2,1) & 2 & (1,2,2) & 2 & (1,2,2) \cr
3 & (0,0,1) & 3 & (1,0,2) & 3 & (0,2,1) &  & \cr
4 & (0,2,0) & & & & & & \cr
5 & (0,0,2) & & & & & & \cr
6 & (1,2,0) & & & & & & \cr
7 & (0,2,1) & & & & & & \cr
8 & (1,0,2) & & & & & & \cr
\end{tabular}
\end{center}
\label{tab:server_state}
\end{table}%

For $i\in\{1,2\}$, let $L^\dag_i(t)$ be the number of customers in Q$_i$ at time $t$. For $i\in\{1,2\}$, define $L_i(t)$ as $L_i(t)=\max\{0, L^\dag_i(t)-1\}$. 
Denote by $(j_1,j_2,j_3)$ the states of the servers: if the server of Q$_1$ is idle, then $j_1=0$ and if it is serving a customer in Q$_1$, then $j_1=1$; if the server of Q$_2$ is idle, then $j_2=0$ and if it is serving a customer in Q$_2$, then $j_2=2$; if the additional server is idle, then $j_3=0$, if it is serving a customer in Q$_1$, then $j_3=1$ and if it is serving a customer in Q$_2$, then $j_3=2$. 
Let $J(t)$ be the server state at time $t$, which is defined as follows: if $L_1(t)=L_2(t)=0$, $J(t)$ takes a value in $S_0=\{1,2,3,4,5,6,7,8\}$ as Table \ref{tab:server_state}; if $L_1(t)>0$ and $L_2(t)=0$, $J(t)$ takes a value in $S_1=\{1,2,3\}$ as Table \ref{tab:server_state}; if $L_1(t)=0$ and $L_2(t)>0$, $J(t)$ takes a value in $S_2=\{1,2,3\}$ as Table \ref{tab:server_state}; if $L_1(t)>0$ and $L_2(t)>0$, $J(t)$ takes a value in $S_+=\{1,2\}$ as Table \ref{tab:server_state}. 
In several states of the servers, a portion of server ability is used ineffectively. For example, when $L_1(t)=0$, $L_2(t)>0$ and $J(t)=3$, there is one customer in Q$_1$ and there is at lest one waiting customer in Q$_2$, but the server of Q$_1$ is idle since the additional server is serving the customer in Q$_1$. This is a reason why the efficiency of the model becomes less than one. 
The process $\{\bY(t)\}=\{((L_1(t),L_2(t)),J(t))\}$ is a 2d-QBD process and it is governed by the infinitesimal generator $Q$ composed of the following block matrices:
\begin{align*}
&A^{(+)}_{-1,0} = \begin{pmatrix}
2\mu_1 & 0 \cr
0 & \mu_1
\end{pmatrix},\  
A^{(+)}_{0,0} = \begin{pmatrix}
-(\lambda+2\mu_1+\mu_2) & 0  \cr
0 & -(\lambda+\mu_1+2\mu_2) \cr
\end{pmatrix},\\ 
&A^{(+)}_{0,-1} = \begin{pmatrix}
\mu_2 & 0 \cr
\mu_2 & \mu_2
\end{pmatrix},\ 
A^{(+)}_{1,0}=\lambda_1 I,\ 
A^{(+)}_{0,1}=\lambda_2 I,\ 
A^{(+)}_{1,1}=A^{(+)}_{-1,1}=A^{(+)}_{1,-1}=A^{(+)}_{-1,-1}=O, 
\end{align*}
\begin{align*}
&A^{(1)}_{0,0} = \begin{pmatrix}
-(\lambda+2\mu_1) & \lambda_2 & 0 \cr
\mu_2 & -(\lambda+2\mu_1+\mu_2) & 0 \cr
\mu_2 & 0 & -(\lambda+\mu_1+\mu_2)
\end{pmatrix},\ 
A^{(1)}_{0,1} = \begin{pmatrix}
0 & 0 \cr
\lambda_2 & 0 \cr
0 & \lambda_2 
\end{pmatrix},\\ 
&A^{(1)}_{-1,0} = \begin{pmatrix}
2\mu_1 & 0 & 0 \cr
0 & 2\mu_1 & 0\cr
0 & 0 & \mu_1 
\end{pmatrix},\ 
A^{(1)}_{0,-1} = \begin{pmatrix}
0 & \mu_2 & 0 \cr
0 & \mu_2 & \mu_2
\end{pmatrix},\\
&A^{(1)}_{1,0}=\lambda_1 I,\ 
A^{(1)}_{1,1}=A^{(1)}_{-1,1}=O,\ A^{(1)}_{1,-1}=A^{(1)}_{-1,-1}=O, 
\end{align*}
\begin{align*}
&A^{(2)}_{0,0} = \begin{pmatrix}
-(\lambda+2\mu_1) & \lambda_1 & 0 \cr
\mu_1 & -(\lambda+\mu_1+2\mu_2) & 0 \cr
\mu_1 & 0 & -(\lambda+\mu_1+\mu_2)
\end{pmatrix},\ 
A^{(2)}_{1,0} = \begin{pmatrix}
0 & 0 \cr
0 & \lambda_1 \cr
\lambda_1 & 0 
\end{pmatrix},\\ 
&A^{(2)}_{0,-1} = \begin{pmatrix}
2\mu_2 & 0 & 0 \cr
0 & 2\mu_2 & 0\cr
0 & 0 & \mu_2 
\end{pmatrix},\ 
A^{(2)}_{-1,0} = \begin{pmatrix}
0 & \mu_1 & \mu_1 \cr
0 & \mu_1 & 0
\end{pmatrix},\\
&A^{(2)}_{1,0}=\lambda_2 I,\ 
A^{(2)}_{1,1}=A^{(2)}_{-1,1}=O,\ A^{(2)}_{1,-1}=A^{(2)}_{-1,-1}=O, 
\end{align*}
where $\lambda=\lambda_1+\lambda_2$; we omit the description of $A^{(0)}_{ij}$ for $i,j\in\mathbb{H}$ since they are not used for evaluating the value of the efficiency of the model. 
%
%

The infinitesimal generator of $\calL^{(+)}$ is given by
\begin{align*}
&A^{(+)}_{*,*} = \begin{pmatrix}
0 & 0 \cr
\mu_2 & -\mu_2 \cr
\end{pmatrix}. 
\end{align*}
Hence, $\calL^{(+)}$ is reducible and has just one irreducible class, which is $\{1\}$. The stationary distribution of $\calL^{(+)}$ is given by $\bpi^{(+)}_{*,*}=\begin{pmatrix} 1 & 0 \end{pmatrix}$ and the mean transition rate vector $\ba^{(+)}=(a^{(+)}_1,a^{(+)}_2)$ is given as $a^{(+)}_1=\lambda_1-2\mu_1$ and $a^{(+)}_2=\lambda_2-\mu_2$. 
The nonzero block matrices of the infinitesimal generator of $\calL^{(1)}$ are given by
\begin{align*}
&A^{(+)}_{*,-1}=A^{(+)}_{0,-1},\ 
A^{(+)}_{*,0} = \begin{pmatrix}
-(\lambda_2+\mu_2) & 0 \cr
0 & -(\lambda_2+2\mu_2) \cr
\end{pmatrix}, \ 
A^{(+)}_{*,1}=\lambda_2 I,\\ 
&A^{(1)}_{*,-1}=A^{(1)}_{0,-1},\ 
A^{(1)}_{*,0} = \begin{pmatrix}
-\lambda_2 & \lambda_2 & 0 \cr
\mu_2 & -(\lambda_2+\mu_2) & 0 \cr
\mu_2 & 0 & -(\lambda_2+\mu_2)
\end{pmatrix},\ 
A^{(1)}_{*,1}=A^{(1)}_{0,1}.
\end{align*}
From the structure of these block matrices, we see that $\calL^{(1)}$ is reducible and has just one irreducible class, which is infinite. 
On the other hand, the nonzero block matrices of the infinitesimal generator of $\calL^{(2)}$ are given by
\begin{align*}
&A^{(+)}_{-1,*}=A^{(+)}_{-1,0},\ 
A^{(+)}_{0,*} = \begin{pmatrix}
-(\lambda_1+2\mu_1) & 0 \cr
\mu_2 & -(\lambda_1+\mu_1+\mu_2) \cr
\end{pmatrix}, \ 
A^{(+)}_{1,*}=\lambda_1 I,\\ 
&A^{(2)}_{-1,*}=A^{(2)}_{-1,0},\ 
A^{(2)}_{0,*} = \begin{pmatrix}
-\lambda_1 & \lambda_1 & 0 \cr
\mu_1 & -(\lambda_1+\mu_1) & 0 \cr
\mu_1 & 0 & -(\lambda_1+\mu_1)
\end{pmatrix},\ 
A^{(2)}_{1,*}=A^{(2)}_{1,0}.
\end{align*}
From the structure of these block matrices, we see that $\calL^{(2)}$ is irreducible. 
By Theorem \ref{th:stability_cond1}, if $a^{(+)}_1<0$, $a^{(+)}_2\ge 0$ and $a^{(2)}_2<0$, $\{\bY(t)\}$ is positive recurrent; if $a^{(+)}_1<0$, $a^{(+)}_2\ge 0$ and $a^{(2)}_2>0$, it is transient. 
When $\mu_1=\mu_2=1$ and $a^{(+)}_1=\lambda_1-2\mu_1<0$, the value of $\lambda_2$ that makes $a^{(2)}$ equal $0$, denoted by $\lambda_2^*$, is given for each value of $\lambda_1$ in Table \ref{tab:table2}, where the efficiency of the model is given by $\rho^*=(\lambda_1/\mu_1+\lambda_2^*/\mu_2)/3$ and the maximum throughput vector by $(\lambda_1,\lambda_2^*)$. 
From the table, we can see how the additional server relieves congestion of Q$_2$ depending on the value of $\lambda_1$. In this case, the efficiency of the model is improved as the value of $\lambda_1$ increases.

\begin{table}[htp]
\caption{The values of $\lambda_2$ that make $a^{(2)}$ equal $0$.}
\begin{center}
\begin{tabular}{c|ccccccccc}
$\lambda_1$ & 1.1 & 1.2 & 1.3 & 1.4 & 1.5 & 1.6 & 1.7 & 1.8 & 1.9 \cr \hline
$\lambda_2^*$ & 1.610 & 1.550 & 1.488 & 1.424 & 1.357 & 1.289 & 1.219 & 1.147 & 1.074 \cr
$\rho^*$ & 0.903 & 0.917 & 0.929 & 0.941 & 0.952 & 0.963 & 0.973 & 0.982 & 0.991
\end{tabular}
\end{center}
\label{tab:table2}
\end{table}%

%
%
\section{Proof of the main theorem} \label{sec:proof} 

%
%
\subsection{Discrete-time 2d-QBD process} \label{sec:dt2dQBD}

In order to prove Theorem \ref{th:stability_cond1}, we use a method developed for analyzing stability of 2d-RRWs \cite{Fayolle89,Fayolle95}. 
Consider the 2d-QBD process $\{\bY(t)\}=\{((L_1(t),L_2(t)),J(t))\}$ defined in Section \ref{sec:model}. Setting the uniformization parameter $\nu<\infty$ so that it satisfies $-q_{((l_1,l_2),j),((l_1,l_2),j)}\le \nu$  for every $((l_1,l_2),j)\in\calS$, we obtain a discrete-time 2d-QBD process from $\{\bY(t)\}$ by  uniformization. We denote the discrete-time 2d-QBD process by $\{\bY_n\}=\{((L_{1,n},L_{2,n}),J_n)\}$. 
$\{\bY_n\}$ is a (discrete-time) Markov chain on the state space $\calS$ whose transition probability matrix $P$ is given in block form as
\[
P=\left( P_{(l_1,l_2),(l'_1,l_2')}; (l_1,l_2),(l_1',l_2')\in\mathbb{Z}_+^2 \right),
\]
where $P_{(l_1,l_2),(l_1',l_2')} = \delta_{l_1,l_1'} \delta_{l_2,l_2'} I + Q_{(l_1,l_2),(l_1',l_2')}/\nu$ and $\delta_{l,l'}$ is the Kronecker delta.  The block matrices of $P$ are, therefore, given in terms of $\bar{A}^{(i)}_{k_1,k_2},\,i\in\{0,1,2,+\},\,k_1,k_2\in\mathbb{H}$, like $Q$, where $\bar{A}^{(i)}_{k_1,k_2} = \delta_{k_1,0} \delta_{k_2,0} I + A^{(i)}_{k_1,k_2}/\nu$ (see expression (\ref{eq:Q_blocks})). 
For $k_1,k_2\in\mathbb{H}$, define $\bar{A}^{(+)}_{*,*}$, $\bar{A}^{(+)}_{*,k_2}$, $\bar{A}^{(+)}_{k_1,*}$, $\bar{A}^{(1)}_{*,k_2}$ and $\bar{A}^{(2)}_{k_1,*}$ in a manner similar to that used for defining $A^{(+)}_{*,*}$, $A^{(+)}_{*,k_2}$, $A^{(+)}_{k_1,*}$, $A^{(1)}_{*,k_2}$ and $A^{(2)}_{k_1,*}$. 
Under Assumption \ref{as:Yt_irreducible}, the Markov chain $\{\bY_n\}$ is irreducible.

Analogously to the case of $\{\bY(t)\}$, we define (discrete-time) Markov chains $\{\bY^{(+)}_n\}$, $\{\bY^{(1)}_n\}$ and $\{\bY^{(2)}_n\}$ corresponding to $\{\bY^{(+)}(t)\}$, $\{\bY^{(1)}(t)\}$ and $\{\bY^{(2)}(t)\}$, respectively. 
For $i\in\{1,2,+\}$, the state space of $\{\bY^{(i)}_n\}=\{((L^{(i)}_{1,n},L^{(i)}_{2,n}),J^{(i)}_n)\}$ is given by $\calS^{(i)}$ and its transition probability matrix $P^{(i)}$ is given in block form as
\[
P^{(i)}=\left( P^{(i)}_{(l_1,l_2),(l'_1,l_2')}; (l_1,l_2),(l_1',l_2')\in\calS^* \right),
\]
where $P^{(i)}_{(l_1,l_2),(l_1',l_2')} = \delta_{l_1,l_1'} \delta_{l_2,l_2'} I + Q^{(i)}_{(l_1,l_2),(l_1',l_2')}/\nu$;  $\calS^*$ is $\mathbb{Z}^2$ if $i=``+"$; it is $\mathbb{Z}\times\mathbb{Z}_+$ if $i=1$; it is $\mathbb{Z}_+\times\mathbb{Z}$ if $i=2$. 
For $\{\bY_n\}$, there are also three induced Markov chains: $\bar{\calL}^{(+)}$, $\bar{\calL}^{(1)}$ and $\bar{\calL}^{(2)}$. $\bar{\calL}^{(+)}$ is given as $\bar{\calL}^{(+)} = \{J^{(+)}_n\}$ and its state space is given by $S_+$. 
$\bar{\calL}^{(1)}$ and $\bar{\calL}^{(2)}$ are given as $\bar{\calL}^{(1)} = \{(L^{(1)}_{2,n},J^{(1)}_n)\}$ and $\bar{\calL}^{(2)} = \{(L^{(2)}_{1,n},J^{(2)}_n)\}$, respectively, and their state spaces are given by  $(\{0\}\times S_1)\cup(\mathbb{N}\times S_+)$ and $(\{0\}\times S_2)\cup(\mathbb{N}\times S_+)$, respectively. $\bar{\calL}^{(1)}$ and $\bar{\calL}^{(2)}$ are ordinary discrete-time QBD processes. 
Under Assumption \ref{as:calL_irreducible}, $\bar{\calL}^{(+)}$ has just one irreducible class and, under Assumption \ref{as:calL12_irreducible}, $\bar{\calL}^{(1)}$ and  $\bar{\calL}^{(2)}$ have at most one irreducible class. 

For $i\in\{1,2,+\}$, $\bar{\calL}^{(i)}$ is a Markov chain obtained from the CTMC $\calL^{(i)}$ by uniformization and the stationary distribution of $\bar{\calL}^{(i)}$ is identical to that of $\calL^{(i)}$. 
Hence, the mean increment vector with respect to $\bar{\calL}^{(+)}$, $\bar{\ba}^{(+)}=(\bar{a}^{(+)}_1,\bar{a}^{(+)}_2)$, is given as 
\begin{align*}
&\bar{a}^{(+)}_1 = \bpi^{(+)}_{*,*} (-\bar{A}^{(+)}_{-1,*}+\bar{A}^{(+)}_{1,*}) \bone = a^{(+)}_1/\nu,\quad 
\bar{a}^{(+)}_2 = \bpi^{(+)}_{*,*} (-\bar{A}^{(+)}_{*,-1}+\bar{A}^{(+)}_{*,1}) \bone = a^{(+)}_2/\nu. 
\end{align*}
Analogously, for $i\in\{1,2\}$, if $\bar{a}^{(+)}_{3-i}<0$, the induced Markov chain $\bar{\calL}^{(i)}$ is positive recurrent and the mean increment vector with respect to $\bar{\calL}^{(i)}$, $\bar{\ba}^{(i)}=(\bar{a}^{(i)}_1,\bar{a}^{(i)}_2)$,  is given as $\bar{\ba}^{(i)}=\ba^{(i)}/\nu$.
For $i\in\{1,2,+\}$, $\bar{\ba}^{(i)}$ is the mean increment vector of the level process of $\{\bY^{(i)}_n\}=\{(\bL^{(i)}_n,J^{(i)}_n)\}$, where $\bL^{(i)}_n=(L^{(i)}_{1,n},L^{(i)}_{2,n})$, and it satisfies, for any $\by\in\calS^{(i)}$,
\begin{align}
& \bar{\ba}^{(i)} 
= \lim_{k\to\infty} \frac{1}{k} \sum_{n=1}^k \mathbb{E}(\bL^{(i)}_n-\bL^{(i)}_{n-1}\,|\,\bY^{(i)}_0=\by)
= \lim_{k\to\infty} \frac{1}{k} \mathbb{E}(\bL^{(i)}_k-\bL^{(i)}_0\,|\,\bY^{(i)}_0=\by). 
\label{eq:a1i_limit}
\end{align}
We use this fact after. 
Since $\{\bY_n\}$ is a Markov chain obtained from the CTMC $\{\bY(t)\}$ by uniformization, $\{\bY(t)\}$ is positive recurrent (resp.\ transient) if and only if $\{\bY_n\}$ is positive recurrent (resp.\ transient). Hence, in order to prove Theorem \ref{th:stability_cond1}, it suffices to prove the following corollary. 
\begin{corollary} \label{co:stability_cond2}
\begin{itemize}
\item[(i)] In the case where $\bar{a}^{(+)}_1<0$ and $\bar{a}^{(+)}_2<0$, the discrete-time 2d-QBD process $\{\bY_n\}$ is positive recurrent if $\bar{a}^{(1)}_1<0$ and $\bar{a}^{(2)}_2<0$, and it is transient if either $\bar{a}^{(1)}_1>0$ or $\bar{a}^{(2)}_2>0$. 
\item[(ii)] In the case where $\bar{a}^{(+)}_1\ge 0$ and $\bar{a}^{(+)}_2<0$, $\{\bY_n\}$ is positive recurrent if $\bar{a}^{(1)}_1<0$, and it is transient if $\bar{a}^{(1)}_1>0$. 
\item[(iii)] In the case where $\bar{a}^{(+)}_1<0$ and $\bar{a}^{(+)}_2\ge 0$, $\{\bY_n\}$ is positive recurrent if $\bar{a}^{(2)}_2<0$, and it is transient if $\bar{a}^{(2)}_2>0$. 
\item[(iv)] If one of $\bar{a}^{(+)}_1$ and $\bar{a}^{(+)}_2$ is positive and the other is non-negative, then $\{\bY_n\}$ is transient.
\end{itemize}
\end{corollary}

%
%
\subsection{Embedded Markov chain} \label{sec:embeddedMC} 

We consider a kind of embedded Markov chain for the discrete-time 2d-QBD process $\{\bY_n\}$.
Let $u_1$, $u_2$ and $u_+$ be positive integers. Let $K_+$, $K_1$ and $K_2$ be positive integers satisfying $K_1>K_+\ge 2$, $K_2>K_+$ and $K_i>u_i$ for $i\in\{1,2,+\}$. Divide the state space $\calS$ into exclusive subsets $\calV_+$, $\calV_1$, $\calV_2$ and $\calV_0$, defined as 
\begin{align*}
&\calV_+ = \{((l_1,l_2),j)\in\calS: l_1\ge K_+,\ l_2\ge K_+\}, \quad 
\calV_1 = \{((l_1,l_2),j)\in\calS: l_1\ge K_1,\ l_2<K_+\}, \\
&\calV_2 = \{((l_1,l_2),j)\in\calS: l_1< K_+,\ l_2\ge K_2\}, \quad 
\calV_0 = \{((l_1,l_2),j)\in\calS: l_1< K_1,\ l_2<K_2\}\setminus\calV_+
\end{align*}
(see Fig.\ \ref{fig:calS_partition}). Define a function $u$ on $\calS$ as 
\[
u(\by) = \left\{ \begin{array}{ll}
u_i & \mbox{if $\by\in\calV_i$ for some $i\in\{1,2,+\}$}, \cr
1 & \mbox{otherwise},
\end{array} \right.
\]
a random sequence $\{\sigma_n\}$ as 
\[
\sigma_0=0,\quad \sigma_{n+1}=\sigma_n+u(\bY_{\sigma_n}),\,n\ge 0,
\]
and a Markov chain $\{\hat{\bY}_n\}=\{((\hat{L}_{1,n},\hat{L}_{2,n}),\hat{J}_n)\}$ as $\hat{\bY}_n = \bY_{\sigma_n},\,n\ge 0$. The process $\{\hat{\bY}_n\}$ is an embedded Markov chain of $\{\bY_n\}$.
\begin{remark} \label{re:embeddedMC} 
If $\hat{\bY}_k=\bY_{\sigma_k}=((L_{1,\sigma_k},L_{2,\sigma_k}),J_{\sigma_k})\in\calV_+$, we have $u(\bY_{\sigma_k})=u_+<K_+$ and $L_{m,\sigma_k}\ge K_+$ for $m\in\{1,2\}$.
Since the level process $\{(L_{1,n},L_{2,n})\}$ is skip free, $\{\bY_n\}=\{((L_{1,n},L_{2,n}),J_n)\}$ does not touch the boundaries of the state space $\calS$, during the time interval $[\sigma_k,\sigma_{k+1}]$. Therefore, in a stochastic sense, $\{\bY_n\}$ behaves just like the Markov chain $\{\bY^{(+)}_n\}$ during that time interval. 
Analogously, for $i\in\{1,2\}$, if $\hat{\bY}_k=\bY_{\sigma_k}\in\calV_i$, then $\{\bY_n\}$ behaves just like the Markov chain $\{\bY^{(i)}_n\}$ during the time interval $[\sigma_k,\sigma_{k+1}]$. 
\end{remark}

\begin{figure}[htbp]
\begin{center}
\setlength{\unitlength}{0.6mm}
\begin{picture}(80,70)(0,0)
\thicklines
\put(0,10){\vector(1,0){70}}
\put(10,0){\vector(0,1){70}}
\put(70,4){\makebox(0,0){\normalsize $l_1$}}
\put(4,68){\makebox(0,0){\normalsize $l_2$}}
\thinlines
\put(30,30){\line(1,0){40}}
\put(30,30){\line(0,1){40}}
\multiput(10,30)(2,0){10}{\line(1,0){1}}
\multiput(30,10)(0,2){10}{\line(0,1){1}}
\put(2,30){\makebox(0,0){\normalsize $K_+$}}
\put(30,5){\makebox(0,0){\normalsize $K_+$}}
\put(50,50){\makebox(0,0){\normalsize $\calV_+$}}
\put(10,40){\line(1,0){20}}
\put(50,10){\line(0,1){20}}
\put(2,40){\makebox(0,0){\normalsize $K_2$}}
\put(50,5){\makebox(0,0){\normalsize $K_1$}}
\put(20,50){\makebox(0,0){\normalsize $\calV_2$}}
\put(60,20){\makebox(0,0){\normalsize $\calV_1$}}
\put(20,20){\makebox(0,0){\normalsize $\calV_0$}}
\end{picture}
\caption{Partition of the state space $\calS$.}
\label{fig:calS_partition}
\end{center}
\end{figure}
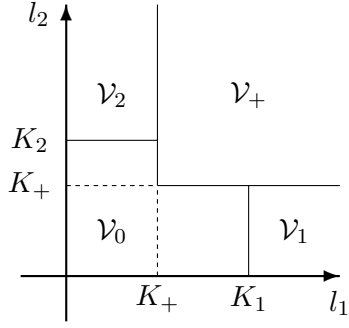

%
To prove Corollary \ref{co:stability_cond2}, we will use the following proposition, which is a modification of Theorem 2.2.4 of \cite{Fayolle95} (also see Theorem 1.4 of \cite{Malyshev81} and Proposition 4.5 of \cite{Bramson08}).

\begin{proposition} \label{pr:Foster2} 
The discrete-time 2d-QBD process $\{\bY_n\}$ is positive recurrent if there exist parameter sets $\{K_+,K_1,K_2\}$ and $\{u_+,u_1,u_2\}$, a positive number $\delta$, a finite subset $\calS_0\subset\calS$ and a lower bounded real function $f$ on $\calS$ such that
\begin{itemize}
\item[(i)] $\mathbb{E}(f(\hat{\bY}_1)-f(\hat{\bY}_0)\,|\,\hat{\bY}_0=\by) \le -\delta$ \ for every $\by\in\calS\setminus\calS_0$, and 
\item[(ii)] $\mathbb{E}(f(\hat{\bY}_1)\,|\,\hat{\bY}_0=\by) < \infty$ \ for every $\by\in\calS_0$. 
\end{itemize}
\end{proposition}

In the case where $u_i$ is set at $1$ for every $i\in\{1,2,+\}$, this proposition is called {\it Foster's criterion}. 
We will also use the following proposition, which is a modification of Theorem 2.2.7 of \cite{Fayolle95} (also see Theorem 1.6 of \cite{Malyshev81}). 
\begin{proposition} \label{pr:Markov_unstable2} 
The discrete-time 2d-QBD process $\{\bY_n\}$ is transient if there exist parameter sets $\{K_+,K_1,K_2\}$ and $\{u_+,u_1,u_2\}$, a real function $f$ on $\calS$ and positive numbers $\delta$, $c$ and $b$ such that, for $\calA=\{\by\in\calS : f(\by)>c \}$,   
\begin{itemize}
\item[(i)] $\calA\ne\emptyset$, 
\item[(ii)] $\mathbb{E}(f(\hat{\bY}_1)-f(\hat{\bY}_0)\,|\,\hat{\bY}_0=\by) \ge \delta$ \ for every $\by\in\calA$, and
\item[(iii)] the inequality $|f(\by')-f(\by)|>b$ implies $\mathbb{P}(\bY_1=\by'\,|\,\bY_0=\by)=0$. 
\end{itemize}
\end{proposition}

In the following subsection, we prepare to construct test function $f$ appeared in the above propositions.

%
%
\subsection{Time averaged increment vectors} \label{sec:cond_mean_increment}

For $i\in\{1,2,+\}$, consider the Markov chain $\{\bY^{(i)}_n\}=\{(\bL^{(i)}_n,J^{(i)}_n)\}$ defined in Subsection \ref{sec:dt2dQBD}, and define, for $\by\in\calS^{(i)}$ and $k\ge 1$, the expectation of the time-averaged increment vector of $\{\bY^{(i)}_n\}$, $\bg^{(i)}_{\by}(k)=(g^{(i)}_{1,\by}(k),g^{(i)}_{2,\by}(k))$, as 
\[
\bg^{(i)}_{\by}(k)
= \mathbb{E}\bigg( \frac{1}{k} \sum_{n=1}^{k} (\bL^{(i)}_n-\bL^{(i)}_{n-1})\,\Big|\,\bY^{(i)}_0=\by \bigg).
\]
From equation (\ref{eq:a1i_limit}), we see that, for $i\in\{1,2,+\}$, if the induced Markov chain $\bar{\calL}^{(i)}$ has a unique stationary distribution, $\bg^{(i)}_{\by}(k)$ satisfies 
\begin{equation}
\lim_{k\to\infty} \bg^{(i)}_{\by}(k) = \bar{\ba}^{(i)}. 
\label{eq:limit_gi}
\end{equation}
Under Assumption \ref{as:calL_irreducible}, $\bar{\calL}^{(+)}=\{J^{(+)}_n\}$ has the unique stationary distribution $\bpi_{*,*}$. For any $\by=((l_1,l_2),j)\in\calS^{(+)}$, $\mathbb{E}(\bL^{(+)}_n-\bL^{(+)}_{n-1}\,|\,\bY^{(+)}_0=\by)=\mathbb{E}(\bL^{(+)}_n-\bL^{(+)}_{n-1}\,|\,J^{(+)}_0=j) $ and the state space $S_+$ of $\{J^{(+)}_n\}$ is finite. Hence, we immediately obtain an approximation for $\bg^{(+)}_{\by}(k)$, as follows. 
\begin{proposition} \label{pr:approximation_g}
For any $\varepsilon>0$, there exists a positive integer $u_+^*$ such that if $k\ge u_+^*$, then for every $\by\in\calS^{(+)}$, 
\begin{equation}
\big| g^{(+)}_{m,\by}(k) - \bar{a}^{(+)}_m \big| < \varepsilon\quad \mbox{for $m\in\{1,2\}$}. 
\label{eq:app_g}
\end{equation}
\end{proposition}

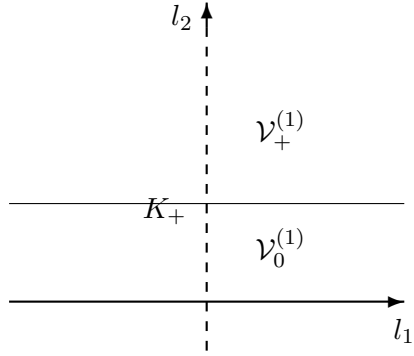
\begin{figure}[htbp]
\begin{center}
\setlength{\unitlength}{0.65mm}
\begin{picture}(85,70)(0,0)
\thicklines
\put(0,10){\vector(1,0){80}}
\put(80,4){\makebox(0,0){\normalsize $l_1$}}
\thicklines
\put(40,66){\vector(0,1){5}}
\multiput(40,0)(0,4){17}{\line(0,1){2}}
\put(35,68){\makebox(0,0){\normalsize $l_2$}}
\thinlines
\put(0,30){\line(1,0){80}}
\put(31.5,28.5){\makebox(0,0){\normalsize $K_+$}}
\put(55,45){\makebox(0,0){\normalsize $\calV^{(1)}_+$}}
\put(55,21){\makebox(0,0){\normalsize $\calV^{(1)}_0$}}
\end{picture}
\caption{Partition of the state space $\calS^{(1)}$.}
\label{fig:tildeS_partition}
\end{center}
\end{figure}

For $i\in\{1,2\}$, let $\{\calV^{(i)}_0, \calV^{(i)}_+\}$ be a partition of the state space $\calS^{(i)}$, defined as
\[
\calV^{(i)}_0 = \{((l_1,l_2),j)\in\calS^{(i)}: l_{3-i}<K_+ \},\quad
\calV^{(i)}_+ = \calS^{(i)}\setminus \calV^{(i)}_0  
\]
 (see Fig.\ \ref{fig:tildeS_partition}).
The following proposition gives approximations for $\bg^{(1)}_{\by}(k)$ and $\bg^{(2)}_{\by}(k)$. 
%
\begin{proposition} \label{pr:approximation_g12} 
Let $\varepsilon$ be an arbitrary positive number and set $u_+$ so that it satisfies, for any $\by\in\calS^{(+)}$, $|g^{(+)}_{m,\by}(u_+)-\bar{a}^{(+)}_m|<\varepsilon/4$ for $m\in\{1,2\}$, which is possible by Proposition \ref{pr:approximation_g}. Furthermore, set $K_+$ so that it satisfies $K_+>u_+$. 
For $i\in\{1,2\}$, $\bg^{(i)}_{\by}(k)$ is approximated as follows.
\begin{itemize}
\item[(i)] When $\bar{a}^{(+)}_{3-i}<0$, there exists a positive integer $u_i^*$ such that if $k\ge u_i^*$, then for every $\by\in\calV^{(i)}_0$, 
\begin{equation}
\big| g^{(i)}_{m,\by}(k) - \bar{a}^{(i)}_m \big| < \varepsilon\quad \mbox{for $m\in\{1,2\}$}. 
\end{equation}
\item[(ii)] When $\bar{a}^{(+)}_{3-i}\ge 0$, there exists a positive integer $u_i^*$ such that if $k\ge u_i^*$, then for every $\by\in\calV^{(i)}_0$, 
\begin{equation}
\big| g^{(i)}_{m,\by}(k) - \bar{a}^{(+)}_m \big| < \varepsilon\quad \mbox{for $m\in\{1,2\}$}. 
\end{equation}

%
\end{itemize}
\end{proposition}

Since the proof of this proposition is elementary, we give it in Appendix \ref{sec:app_proof_pro3_4}.

%
%
\subsection{Proof of Corollary \ref{co:stability_cond2}} \label{sec:proof_subsec}

Using a linear function on $\mathbb{R}^2$, we construct a test function and apply Propositions \ref{pr:Foster2} and \ref{pr:Markov_unstable2} to the discrete-time 2d-QBD process $\{\bY_n\}$.  
Denote by $\langle\bx_1,\bx_2\rangle$ the inner product of vectors $\bx_1,\bx_2\in\mathbb{R}^2$. The linear function  of $\bx\in\mathbb{R}^2$ is given by $\langle\bx,\bw\rangle$, where $\bw$ is a given vector. 
For $\by\in\calS$, define the one-step mean increment vector of the embedded Markov chain $\{\hat{\bY}_n\}=\{(\hat{\bL}_n,\hat{J}_n)\}$, $\hat{\balpha}_{\by}=(\hat{\alpha}_{1,\by},\hat{\alpha}_{2,\by})$, as
\[
\hat{\balpha}_{\by} 
= \mathbb{E}(\hat{\bL}_1-\hat{\bL}_0\,|\,\hat{\bY}_0=\by)
= \mathbb{E}(\bL_{u(\by)}-\bL_0\,|\,\bY_0=\by). 
\]
The following proposition corresponds to Condition B and Theorem 2.1 of \cite{Malyshev81} (also see Condition B and Theorem 4.3.4 of \cite{Fayolle95}). 
%
\begin{proposition} \label{pr:positive_recurrence} 
The discrete-time 2d-QBD process $\{\bY_n\}$ is positive recurrent if there exist parameter sets $\{K_+,K_1,K_2\}$ and $\{u_+,u_1,u_2\}$, a positive vector $\bw=(w_1,w_2)$ and a positive number $\delta$ such that, for every $\by\in\calS\setminus\calV_0$, $\langle\hat{\balpha}_{\by},\bw\rangle\le -\delta$. 
\end{proposition}
\begin{proof}
We prove this proposition by Proposition \ref{pr:Foster2}. 
Let $\bw=(w_1,w_2)$ be a positive vector satisfying the condition of the proposition and consider the following function as a test function:
\begin{equation}
f(\by) = f((\bl,j)) = \langle \bl,\bw \rangle,\quad \by=(\bl,j)=((l_1,l_2),j)\in\calS. 
\end{equation}
This function $f$ takes nonnegative values on $\calS$ and hence, it is lower bounded. 
Since $\langle\bx,\bw\rangle$ is linear in $\bx\in\mathbb{R}^2$, we have, for every $\by\in\calS\setminus\calV_0$, 
\begin{align}
\mathbb{E}(f(\hat{\bY}_1)-f(\hat{\bY}_0)\,|\,\hat{\bY}_0=\by) 
= \langle \mathbb{E}(\hat{\bL}_1-\hat{\bL}_0\,|\,\hat{\bY}_0=\by),\bw\rangle 
= \langle\hat{\balpha}_{\by},\bw\rangle
\le -\delta, 
\end{align}
where $\calV_0$ is finite. 
By the definition of the embedded Markov chain $\{\hat{\bY}_n\}$, if $\by=((l_1,l_2),j)\in\calV_0$, then $u(\by)=1$ and we have 
\begin{align}
\mathbb{E}(f(\hat{\bY}_1)\,|\,\hat{\bY}_0=\by) 
=\mathbb{E}(\langle\bL_1,\bw\rangle\,|\,\bY_0=\by\big)
\le K_1 w_1+K_2 w_2<\infty, 
\end{align}
where we use the fact that $\{\bL_n\}$ is skip free.
This completes the proof. 
\end{proof}

\begin{proof}[Proof of Corollary \ref{co:stability_cond2} (positive recurrence)]
First, we consider the case where $\bar{a}^{(+)}_1<0$, $\bar{a}^{(+)}_2<0$, $\bar{a}^{(1)}_1<0$ and $\bar{a}^{(2)}_2<0$. Set $\bw=(1,1)>\bzero$, then we have $\langle\bar{\ba}^{(+)},\bw\rangle=\bar{a}^{(+)}_1+\bar{a}^{(+)}_2<0$, $\langle\bar{\ba}^{(1)},\bw\rangle=\bar{a}^{(1)}_1<0$ and $\langle\bar{\ba}^{(2)},\bw\rangle=\bar{a}^{(2)}_2<0$. 
Set positive numbers $\delta$ and $\varepsilon$ so that they satisfy 
\begin{align}
-(\delta + \varepsilon) 
= \max\{ \langle\bar{\ba}^{(+)},\bw\rangle,\ \langle\bar{\ba}^{(1)},\bw\rangle,\ \langle\bar{\ba}^{(2)},\bw\rangle \}
<0. 
\end{align}
Set positive integer $u_+$ so that it satisfies, for every $\by\in\calS^{(+)}$ and for every $m\in\{1,2\}$, $|g^{(+)}_{m,\by}(u_+)-\bar{a}^{(+)}_m|<\varepsilon/8$. It is possible by Proposition \ref{pr:approximation_g}. Set positive integer $K_+$ so that it satisfies $K_+>u_+$. 
Furthermore, for $i\in\{1,2\}$, set positive integer $u_i$ so that it satisfies, for every $\by\in\calV^{(i)}_0$ and for every $m\in\{1,2\}$, $|g^{(i)}_{m,\by}(u_i)-\bar{a}^{(i)}_m|<\varepsilon/2$. It is possible by Proposition \ref{pr:approximation_g12}. For $i\in\{1,2\}$, set positive integer $K_i$ so that it satisfies $K_i>\max\{u_i, K_+\}$.
Note that we have, for every $\by\in\calV_+\subset\calS^{(+)}$, $\hat{\balpha}_{\by}/u_+=\bg^{(+)}_{\by}(u_+)$ and, for $i\in\{1,2\}$ and for every $\by\in\calV_i\subset\calV^{(i)}_0$, $\hat{\balpha}_{\by}/u_i=\bg^{(i)}_{\by}(u_i)$ (see Remark \ref{re:embeddedMC}). 
Hence, we obtain, for $\by\in\calV_+$, 
\begin{align}
\langle\hat{\balpha}_{\by}/u_+,\bw\rangle 
= \langle\bar{\ba}^{(+)},\bw\rangle + \langle\bg^{(+)}_{\by}(u_+)-\bar{\ba}^{(+)},\bw\rangle
\le -(\delta+\varepsilon)+\varepsilon/8+\varepsilon/8
\le -\delta
\le -\delta/u_+
\end{align}
and, for $i\in\{1,2\}$ and for every $\by\in\calV_i$,
\begin{align}
\langle\hat{\balpha}_{\by}/u_i,\bw\rangle 
= \langle\bar{\ba}^{(i)},\bw\rangle + \langle\bg^{(i)}_{\by}(u_i)-\bar{\ba}^{(i)},\bw\rangle
\le -(\delta+\varepsilon)+\varepsilon/2+\varepsilon/2
= -\delta
\le -\delta/u_i.
\end{align}
As a result, by Proposition \ref{pr:positive_recurrence}, the discrete-time 2d-QBD process $\{\bY_n\}$ is positive recurrent.

Next, we consider the case where $\bar{a}^{(+)}_1\ge 0$, $\bar{a}^{(+)}_2<0$ and $\bar{a}^{(1)}_1<0$. Set $\bw=(-\bar{a}^{(+)}_2,1+\bar{a}^{(+)}_1)>\bzero$, then we have $\langle\bar{\ba}^{(+)},\bw\rangle=\bar{a}^{(+)}_2<0$ and $\langle\bar{\ba}^{(1)},\bw\rangle=-\bar{a}^{(+)}_2 \bar{a}^{(1)}_1<0$. 
Set positive numbers $\delta$ and $\varepsilon$ so that they satisfy 
\begin{align}
-(\delta + \varepsilon) 
= \max\{ \langle\bar{\ba}^{(+)},\bw\rangle,\ \langle\bar{\ba}^{(1)},\bw\rangle \}
<0. 
\end{align}
Set positive integer $u_+$ so that it satisfies, for every $\by\in\calS^{(+)}$ and for every $m\in\{1,2\}$, $|g^{(+)}_{m,\by}(u_+)-\bar{a}^{(+)}_m|<\varepsilon/8$ and positive integer $K_+$ so that it satisfies $K_+>u_+$. 
Furthermore, set positive integers $u_1$ and $u_2$ so that they satisfy, for every $\by\in\calV^{(1)}_0$ and for every $m\in\{1,2\}$, $|g^{(1)}_{m,\by}(u_1)-\bar{a}^{(1)}_m|<\varepsilon/2$ and, for every $\by\in\calV^{(2)}_0$ and for every $m\in\{1,2\}$, $|g^{(2)}_{m,\by}(u_2)-\bar{a}^{(+)}_m|<\varepsilon/2$, respectively. For $i\in\{1,2\}$, set positive integer $K_i$ so that it satisfies $K_i>\max\{u_i, K_+\}$.
Then, we have, for every $i\in\{1,2,+\}$ and for every $\by\in\calV_i$, $\langle\hat{\balpha}_{\by}/u_i,\bw\rangle=\langle\bg^{(i)}_{\by}(u_i),\bw\rangle\le -\delta\le-\delta/u_i$.
As a result, by Proposition \ref{pr:positive_recurrence}, $\{\bY_n\}$ is positive recurrent.

The proof for the case where $\bar{a}^{(+)}_1< 0$, $\bar{a}^{(+)}_2\ge 0$ and $\bar{a}^{(2)}_2<0$ is analogous to the above case. This completes the proof of the corollary in the case where $\{\bY_n\}$ is positive recurrent. 
\end{proof}

%
For a vector $\bw\in\mathbb{R}^2$ and real number $c\in\mathbb{R}$, define a subset of $\calS$, $\calA_{\bw,c}$, as
\[
\calA_{\bw,c} = \{\by=(\bl,j)\in\calS: \langle\bl,\bw\rangle>c\}, 
\]
and an index set $\scrI_{\bw,c}$ as 
\[
\scrI_{\bw,c} = \{i\in\{0,1,2,+\}: \calA_{\bw,c}\cap\calV_i\ne\emptyset\}.
\]
The following proposition corresponds to Condition B$'$ and Theorem 2.1 of \cite{Malyshev81} (also see Condition B$'$ and Theorem 4.3.4 of \cite{Fayolle95}). 
%
\begin{proposition} \label{pr:transience} 
The discrete-time 2d-QBD process $\{\bY_n\}$ is transient if there exist parameter sets $\{K_+,K_1,K_2\}$ and $\{u_+,u_1,u_2\}$, a nonzero vector $\bw=(w_1,w_2)$, a real number $c$ and a positive number $\delta$ such that $\calA_{\bw,c}\ne\emptyset$ and, for every $i\in\scrI_{\bw,c}$ and for every $\by\in\calV_i$, $\langle\hat{\balpha}_{\by},\bw\rangle\ge \delta$. 
\end{proposition}
\begin{proof}
We prove this proposition by Proposition \ref{pr:Markov_unstable2}. 
Let $\bw=(w_1,w_2)$ and $c$ be a real vector and real number satisfying the condition of the proposition. Consider the following test function:
\begin{equation}
f(\by) = f((\bl,j)) = \langle \bl,\bw \rangle,\quad \by=(\bl,j)\in\calS. 
\end{equation}
Then, we have, for every $\by\in\calA_{\bw,c}\subset\cup_{i\in\scrI_{\bw,c}} \calV_i$, 
\begin{align}
\mathbb{E}(f(\hat{\bY}_1)-f(\hat{\bY}_0)\,|\,\hat{\bY}_0=\by) 
&= \langle\hat{\balpha}_{\by},\bw\rangle
\ge \delta.  
\end{align}
Since the process $\{\bL_n\}$ is skip free, we have 
\begin{align}
|f(\bY_1)-f(\bY_0)| 
&=| \langle \bL_1-\bL_0,\bw \rangle |
\le |w_1|+|w_2|. 
\end{align} 
Hence, for every $\by,\by'\in\calS$, if $|f(\by')-f(\by)|>|w_1|+|w_2|$, then $\mathbb{P}(\bY_1=\by'\,|\,\bY_0=\by)=0$. 
This completes the proof. 
\end{proof}

%
\begin{proof}[Proof of Corollary \ref{co:stability_cond2} (transience)]
First, we consider the case where one of $\bar{a}^{(+)}_1$ and $\bar{a}^{(+)}_2$ is positive and the other is non-negative. Set $\bw=(1,1)$, then we have $\langle\bar{\ba}^{(+)},\bw\rangle=\bar{a}^{(+)}_1+\bar{a}^{(+)}_2>0$. 
Set positive numbers $\delta$ and $\varepsilon$ so that they satisfy 
\begin{align}
\delta + \varepsilon = \langle\bar{\ba}^{(+)},\bw\rangle > 0. 
\end{align}
Set positive integer $u_+$ so that it satisfies, for every $\by\in\calS^{(+)}$ and for every $m\in\{1,2\}$, $|g^{(+)}_{m,\by}(u_+)-\bar{a}^{(+)}_m|<\varepsilon/8$ and positive integer $K_+$ so that it satisfies $K_+>u_+$. 
For $i\in\{1,2\}$, set positive integer $u_i$ so that it satisfies, for every $\by\in\calV^{(i)}_0$ and for every $m\in\{1,2\}$, $|g^{(i)}_{m,\by}(u_i)-\bar{a}^{(+)}_m|<\varepsilon/2$, and  set positive integer $K_i$ so that it satisfies $K_i>\max\{u_i, K_+\}$.
Set $c=\max\{K_1,K_2\}+K_+$. Then, 
\begin{align}
&\calA_{\bw,c}
=\{\by=((l_1,l_2),j)\in\calS: \langle(l_1,l_2),\bw\rangle=l_1+l_2>c\}
\ne \emptyset
\end{align}
and we have $\scrI_{\bw,c}=\{1,2,+\}$.
We have, for every $\by\in\calV_+\subset\calS^{(+)}$, 
\begin{align}
\langle\hat{\balpha}_{\by}/u_+,\bw\rangle 
= \langle\bar{\ba}^{(+)},\bw\rangle + \langle\bg^{(+)}_{\by}(u_+)-\bar{\ba}^{(+)},\bw\rangle
\ge \delta+\varepsilon-(\varepsilon/8+\varepsilon/8)
\ge \delta
\ge \delta/u_+
\end{align}
and, for $i\in\{1,2\}$ and for every $\by\in\calV_i\subset\calV^{(i)}_0$,
\begin{align}
\langle\hat{\balpha}_{\by}/u_i,\bw\rangle 
= \langle\bar{\ba}^{(+)},\bw\rangle + \langle\bg^{(i)}_{\by}(u_i)-\bar{\ba}^{(+)},\bw\rangle
\ge \delta+\varepsilon-(\varepsilon/2+\varepsilon/2)
= \delta
\ge \delta/u_i.
\end{align}
As a result, by Proposition \ref{pr:transience}, the discrete-time 2d-QBD process $\{\bY_n\}$ is transient.

Next, we consider the case where $\bar{a}^{(+)}_2<0$ and $\bar{a}^{(1)}_1>0$; $\bar{a}^{(+)}_1$ may take any value.
Set $\bw=(-\bar{a}^{(+)}_2,-(1+|\bar{a}^{(+)}_1|))$, then we have $\langle\bar{\ba}^{(+)},\bw\rangle=-\bar{a}^{(+)}_2+|\bar{a}^{(+)}_1 \bar{a}^{(+)}_2|-\bar{a}^{(+)}_1 \bar{a}^{(+)}_2>0$ and $\langle\bar{\ba}^{(1)},\bw\rangle=-\bar{a}^{(+)}_2 \bar{a}^{(1)}_1>0$. 
Set positive numbers $\delta$ and $\varepsilon$ so that they satisfy 
\begin{align}
\delta + \varepsilon = \min\{\langle\bar{\ba}^{(+)},\bw\rangle,\langle\bar{\ba}^{(1)},\bw\rangle\} > 0. 
\end{align}
Set positive integer $u_+$ so that it satisfies, for every $\by\in\calS^{(+)}$ and for every $m\in\{1,2\}$, $|g^{(+)}_{m,\by}(u_+)-\bar{a}^{(+)}_m|<\varepsilon/8$ and positive integer $K_+$ so that it satisfies $K_+>u_+$. 
Furthermore, set positive integer $u_1$ so that it satisfies, for every $\by\in\calV^{(1)}_0$ and for every $m\in\{1,2\}$, $|g^{(1)}_{m,\by}(u_1)-\bar{a}^{(1)}_m|<\varepsilon/2$ and positive integer $K_1$ so that it satisfies $K_1>\max\{u_1, K_+\}$. 
Set $u_2$ at a sufficiently large positive integer, for example, at $u_1$, and positive integer $K_2$ so that it satisfies $K_2>\max\{u_2, K_+\}$.
Set $c=|\bar{a}^{(+)}_2| K_1$. Then, 
\begin{align}
&\calA_{\bw,c}
=\{\by=((l_1,l_2),j)\in\calS: \langle(l_1,l_2),\bw\rangle=|\bar{a}^{(+)}_2| l_1-(1+|\bar{a}^{(+)}_1|)l_2>c\}
\ne \emptyset
\end{align}
and $\scrI_{\bw,c}=\{1,+\}$.
We have, for $i\in\{1,+\}$ and for every $\by\in\calV_i\subset\calS^{(i)}$, $\langle\hat{\balpha}_{\by}/u_i,\bw\rangle = \langle\bg^{(i)}_{\by}(u_i),\bw\rangle \ge \delta/u_i$.
As a result, by Proposition \ref{pr:transience}, $\{\bY_n\}$ is transient.

The proof for the case where $\bar{a}^{(+)}_1<0$ and $\bar{a}^{(2)}_2>0$ is analogous to the above case. This completes the proof of the corollary in the case where $\{\bY_n\}$ is transient. 
\end{proof}

%
%
\section{Concluding remarks} \label{sec:concluding}

%
For stability analysis of multiple-queue models and queueing networks, a method that can handle multidimensional QBD processes is desired. The notion of induced Markov chain and that of mean increment vector can also be applied to discrete-time multidimensional QBD processes and a certain result has been obtained in \cite{Ozawa15}. 
However, in a discrete-time multidimensional QBD process, several induced Markov chains are also discrete-time multidimensional QBD processes and, in order to evaluate the mean increment vectors, we need the stationary distributions of the multidimensional QBD processes. For example, in a 3d-QBD process, one induced Markov chain is a finite Markov chain, three induced Markov chains are ordinary discrete-time QBD processes and the other three induced Markov chains are discrete-time 2d-QBD processes. 
In general, it is very difficult to obtain the stationary distribution of a multidimensional QBD process if the dimension of the level process is greater than or equal to $2$. 
At present, we have no good ideas to overcome that difficulty, and it is left as a further study.

%

%
%
%

%
%
\appendix

%
%
\section{Setup time model with MAPs and PH-distributions} \label{sec:setup_MAPPH}

For $i\in\{1,2\}$,  the arrival process of class-$i$ customers is given by the MAP with representation $(C_i,D_i)$, the service time distribution of them by the PH-distribution with representation $(U_i, \bbeta_i)$ and the distribution of setup times for class-$i$ customer's service by the PH-distribution with representation $(U^{set}_i, \bbeta^{set}_i)$. 
For $i\in\{1,2\}$, define $\bu_i$ and $\bu^{set}_i$ as $\bu_i=-U_i \bone$ and $\bu^{set}_i=-U^{set}_i\bone$. 
In the priority queue with setup times, the nonzero block matrices of the infinitesimal generator $Q$ are given as follows.
\begin{align*}
&A^{(+)}_{-1,0} = I\otimes I\otimes
\begin{pmatrix}
\bu_1 \bbeta_1 & O & O & O \cr
O & O & O & O \cr
O & O & O & O \cr
O & O & O & O 
\end{pmatrix},\  
A^{(+)}_{0,0} = C_1\oplus C_2\oplus \begin{pmatrix}
U_1 & O & O & O \cr
\bu^{set}_1 \bbeta_1 & U_1^{set} & O & O \cr
O & O & U_2 & O \cr
O & O & \bu^{set}_2 \bbeta_2 & U^{set}_2
\end{pmatrix}, \\
&A^{(+)}_{0,-1} = I\otimes I\otimes 
\begin{pmatrix}
O & O & O & O \cr
O & O & O & O \cr
O & \bu_2 \bbeta^{set}_1 & O & O \cr
O & O & O & O 
\end{pmatrix},\ 
A^{(+)}_{1,0}=D_1\otimes I\otimes I,\ 
A^{(+)}_{0,1}=I\otimes D_2\otimes I,\ 
%
\end{align*}
\begin{align*}
&A^{(1)}_{-1,0} = I\otimes I\otimes 
\begin{pmatrix}
\bu_1\bbeta_1 & O \cr
O & O  
\end{pmatrix},\ 
A^{(1)}_{0,0} = C_1\oplus C_2\oplus 
\begin{pmatrix}
U_1 & O \cr
\bu^{set}_1 \bbeta_1 & U^{set}_1  
\end{pmatrix},\ 
A^{(1)}_{1,0}=D_1\otimes I\otimes I, \\ 
&A^{(1)}_{0,1} = I\otimes D_2\otimes 
\begin{pmatrix}
I & O & O & O \cr
O & I & O & O  
\end{pmatrix},\ 
A^{(2)}_{1,0} =D_1\otimes I\otimes 
\begin{pmatrix}
O & O & I & O \cr
O & O & O & I  
\end{pmatrix},\\
%
%
&A^{(1)}_{0,-1} = I\otimes I\otimes 
\begin{pmatrix}
O & O \cr
O & O \cr
O & \bu_2 \bbeta^{set}_1 \cr
O & O  
\end{pmatrix},\ 
A^{(2)}_{-1,0} = I\otimes I\otimes 
\begin{pmatrix}
O & \bu_1 \bbeta^{set}_2 \cr
O & O \cr
O & O \cr
O & O  
\end{pmatrix},
\end{align*}
\begin{align*}
&A^{(2)}_{0,0} = C_1\oplus C_2\oplus 
\begin{pmatrix}
U_2 & O \cr
\bu^{set}_2 \bbeta_2 & U^{set}_2 
\end{pmatrix},\ 
A^{(2)}_{0,-1} = I\otimes I\otimes 
\begin{pmatrix}
\bu_2 \bbeta_2 & O \cr
O & O  
\end{pmatrix},\ 
A^{(2)}_{0,1}=I\otimes D_2\otimes I,\ 
%
\end{align*} 
\begin{align*}
&A^{(0)}_{-1,0} = I\otimes I\otimes 
\begin{pmatrix}
\bu_1 \cr \bzero  
\end{pmatrix},\ 
A^{(0)}_{0,0} = C_1\oplus C_2,\ 
A^{(0)}_{0,-1} = I\otimes I\otimes 
\begin{pmatrix}
\bu_2 \cr \bzero  
\end{pmatrix},\\ 
&A^{(0)}_{1,0} = D_1\otimes I\otimes 
\begin{pmatrix}
\bzero^\top & \bbeta^{set}_1  
\end{pmatrix},\ 
A^{(0)}_{0,1} = I\otimes D_2\otimes 
\begin{pmatrix}
\bzero^\top & \bbeta^{set}_2  
\end{pmatrix}.
%
%
\end{align*}

%
%
\section{Proof of Proposition \ref{pr:approximation_g12}} \label{sec:app_proof_pro3_4}

\begin{proof}[Proof of Proposition \ref{pr:approximation_g12}]
First, we consider the case where $\bar{a}^{(+)}_2<0$. In this case, the induced Markov chain $\bar{\calL}^{(1)}$ has just one irreducible class, and the unique stationary distribution $\bpi^{(1)}_*$ exists. 
%
%
%
%
%
Furthermore, $\bg^{(1)}_{\by}(k)$ satisfies, for any $\by=((l_1,l_2),j)\in\calV^{(1)}_0$, 
\begin{equation}
\bg^{(1)}_{\by}(k) 
= \frac{1}{k} \sum_{n=1}^k \mathbb{E}\big(\bL^{(1)}_n-\bL^{(1)}_{n-1}\,|\,(L^{(1)}_{2,0},J^{(1)}_0)=(l_2,j)\big), 
\end{equation}
and the set $\{(l_2',j'): \mbox{$((l_1',l_2'),j')\in\calV^{(1)}_0$ for some $l_1'\in\mathbb{Z}$}\}$ is finite. Hence, from equation (\ref{eq:limit_gi}), we obtain statement (i) of the proposition for $\bg^{(1)}_{\by}(k) $.  In the case where $\bar{a}^{(+)}_1<0$, an analogous result holds for $\bg^{(2)}_{\by}(k) $. 

Next, assuming $\bar{a}^{(+)}_2\ge 0$, we consider the case where $\bar{\calL}^{(1)}$ has no irreducible classes or it has just one irreducible class (see Assumption \ref{as:calL12_irreducible}). In this case, any state of $\bar{\calL}^{(1)}$ is transient or null recurrent and we have, for any $(l_2,j),(l_2',j')\in(\{0\}\times S_1)\cup(\mathbb{N}\times S_+)$, 
\begin{equation}
\lim_{n\to\infty} \mathbb{P}\big((L^{(1)}_{2,n},J^{(1)}_n)=(l_2',j')\,|\,(L^{(1)}_{2,0},J^{(1)}_0)=(l_2,j)\big) = 0. 
\end{equation}
Define a function $u^{(1)}$ on $\calS^{(1)}$ as 
\[
u^{(1)}(\by) = \left\{ \begin{array}{ll}
u_+ & \mbox{if $\by\in\calV^{(1)}_+$}, \cr
1 & \mbox{otherwise},
\end{array} \right.
\]
and a random sequence $\{\sigma^{(1)}_n\}$ as 
\[
\sigma^{(1)}_0=0,\quad \sigma^{(1)}_{n+1}=\sigma^{(1)}_n+u^{(1)}(\bY^{(1)}_{\sigma^{(1)}_n}),\,n\ge 0.
\]
For $\by=((l_1,l_2),j)\in\calV^{(1)}_0$, we obtain, by the definition of $\bg^{(1)}_{\by}(k)$, 
\begin{align}
\bg^{(1)}_{\by}(k) 
&= \frac{1}{k} \mathbb{E}\Big(\sum_{n=0}^{k} 1(\sigma^{(1)}_n\le k<\sigma^{(1)}_{n+1}) (\bL^{(1)}_{k}-\bL^{(1)}_0)\,\Big|\,\bY^{(1)}_0=\by\Big) \cr
&\ = \frac{1}{k} \mathbb{E}\Big(\sum_{n=0}^{k} 1(\sigma^{(1)}_n\le k<\sigma^{(1)}_{n+1}) \Big( \sum_{l=1}^n (\bL^{(1)}_{\sigma^{(1)}_l}-\bL^{(1)}_{\sigma^{(1)}_{l-1}})+(\bL^{(1)}_{k}-\bL^{(1)}_{\sigma^{(1)}_n}) \Big)\,\Big|\,\bY^{(1)}_0=\by\Big) \cr
&\ = \frac{1}{k} \sum_{l=1}^{k} \mathbb{E}\big(1(\sigma^{(1)}_l\le k) (\bL^{(1)}_{\sigma^{(1)}_l}-\bL^{(1)}_{\sigma^{(1)}_{l-1}})\,|\,\bY^{(1)}_0=\by\big) \cr
&\qquad\qquad + \frac{1}{k} \sum_{n=0}^{k} \mathbb{E}\big(1(\sigma^{(1)}_n\le k<\sigma^{(1)}_{n+1})(\bL^{(1)}_{k}-\bL^{(1)}_{\sigma^{(1)}_n})\,|\,\bY^{(1)}_0=\by\big), 
\end{align}
where $1(\cdot)$ is an indicator function and we use the fact that $\sigma^{(1)}_n>k$ for $n>k$. We have 
\begin{align}
&\mathbb{E}\big(1(\sigma^{(1)}_l\le k) (\bL^{(1)}_{\sigma^{(1)}_l}-\bL^{(1)}_{\sigma^{(1)}_{l-1}})\,|\,\bY^{(1)}_0=\by\big) \cr
%
%
&\quad= \sum_{\by'\in\calS^{(1)}} \mathbb{E}(\bL^{(1)}_{u^{(1)}(\by')}-\bL^{(1)}_0\,|\,\bY^{(1)}_0=\by')\,\mathbb{P}(\sigma^{(1)}_l\le k,\,\bY^{(1)}_{\sigma^{(1)}_{l-1}}=\by'\,|\,\bY^{(1)}_0=\by) \cr
%
%
&\quad= \sum_{\by'\in\calV^{(1)}_0} \mathbb{E}(\bL^{(1)}_1-\bL^{(1)}_0\,|\,\bY^{(1)}_0=\by')\,\mathbb{P}(\sigma^{(1)}_l\le k,\,\bY^{(1)}_{\sigma^{(1)}_{l-1}}=\by'\,|\,\bY^{(1)}_0=\by) \cr
&\qquad\qquad + \sum_{\by'\in\calV^{(1)}_+} u_+\,\bg^{(+)}_{\by'}(u_+)\,\mathbb{P}(\sigma^{(1)}_l\le k,\,\bY^{(1)}_{\sigma^{(1)}_{l-1}}=\by'\,|\,\bY^{(1)}_0=\by), 
\end{align}
where we use the fact that $\frac{1}{u_+}\mathbb{E}(\bL^{(1)}_{u_+}-\bL^{(1)}_0\,|\,\bY^{(1)}_0=\by')=\bg^{(+)}_{\by'}(u_+)$ for $\by'\in\calV^{(1)}_+\subset\calS^{(+)}$ (see Remark \ref{re:embeddedMC}). 
$\bg^{(1)}_{\by}(k)$ is, therefore, represented as 
\begin{equation}
\bg^{(1)}_{\by}(k) = \bphi^a_{\by}(k) + \bphi^b_{\by}(k) + \bphi^c_{\by}(k), 
\label{eq:gA_phiA123}
\end{equation}
where $\bphi^a_{\by}(k)=(\phi^a_{1,\by}(k),\phi^a_{2,\by}(k))$, $\bphi^b_{\by}(k)=(\phi^b_{1,\by}(k),\phi^b_{2,\by}(k))$ and $\bphi^c_{\by}(k)=(\phi^c_{1,\by}(k),\phi^c_{2,\by}(k))$ are given as 
\begin{align*}
&\bphi^a_{\by}(k) = \frac{1}{k} \sum_{l=1}^{k} \sum_{\by'\in\calV^{(1)}_0} \mathbb{E}(\bL^{(1)}_1-\bL^{(1)}_0\,|\,\bY^{(1)}_0=\by')\,\mathbb{P}(\sigma^{(1)}_l\le k,\,\bY^{(1)}_{\sigma^{(1)}_{l-1}}=\by'\,|\,\bY^{(1)}_0=\by), \cr
&\bphi^b_{\by}(k) = \frac{1}{k} \sum_{l=1}^{k} \sum_{\by'\in\calV^{(1)}_+} u_+\,\bg^{(+)}_{\by'}(u_+)\,\mathbb{P}(\sigma^{(1)}_l\le k,\,\bY^{(1)}_{\sigma^{(1)}_{l-1}}=\by'\,|\,\bY^{(1)}_0=\by), \cr
&\bphi^c_{\by}(k) = \frac{1}{k} \sum_{n=0}^{k} \mathbb{E}\big(1(\sigma^{(1)}_n\le k<\sigma^{(1)}_{n+1})(\bL^{(1)}_{k}-\bL^{(1)}_{\sigma^{(1)}_n})\,|\,\bY^{(1)}_0=\by\big). 
\end{align*}

Define $\calV_0^*$ as $\calV_0^*=\{(l_2',j'): \mbox{$((l_1',l_2'),j')\in\calV^{(1)}_0$ for some $l_1'\in\mathbb{Z}$}\}$. This $\calV_0^*$ is a finite subset of the state space of $\bar{\calL}^{(1)}=\{(L^{(1)}_{2,n},J^{(1)}_n)\}$. Since the process $\{\bL^{(1)}_n\}$ is skip free, we obtain, for $m\in\{1,2\}$, 
\begin{align}
|\phi^a_{m,\by}(k)| 
&\le \frac{1}{k} \sum_{l=1}^{k}\ \sum_{\by'\in\calV^{(1)}_0} \mathbb{P}(\sigma^{(1)}_{l-1}+u^{(1)}(\by')\le k,\,\bY^{(1)}_{\sigma^{(1)}_{l-1}}=\by'\,|\,\bY^{(1)}_0=\by) \cr
%
%
&\le \frac{1}{k} \sum_{l=0}^{k-1} \mathbb{P}(\bY^{(1)}_l\in\calV^{(1)}_0\,|\,\bY^{(1)}_0=\by), 
%
\label{eq:phiA1_g^{(1)}}
\end{align}
where $\mathbb{P}(\bY^{(1)}_l\in\calV^{(1)}_0\,|\,\bY^{(1)}_0=\by) = \mathbb{P}\big((L^{(1)}_{2,l},J^{(1)}_l)\in\calV_0^*\,|\,(L^{(1)}_{2,0},J^{(1)}_0)=(l_2,j)\big)$. 
Since $\calV_0^*$ is finite and any state of $\bar{\calL}^{(1)}$ is transient or null recurrent, there exists a positive integer $u_{1a}^*$ such that if $k\ge u_{1a}^*$, then for every $\by\in\calV^{(1)}_0$ and $m\in\{1,2\}$, $|\phi^a_{m,\by}(k)|<\varepsilon/4$. 
%
Since $\sigma^{(1)}_{n+1}-\sigma^{(1)}_n\le u_+$ for any $n\ge 0$ and $\{\bL^{(1)}_n\}$ is skip free, we have, for every $\by\in\calV^{(1)}_0$ and for $m\in\{1,2\}$, 
\begin{align}
|\phi^c_{m,\by}(k)| 
&\le \frac{1}{k} \sum_{n=0}^{k} \mathbb{E}\big(1(\sigma^{(1)}_n\le k<\sigma^{(1)}_{n+1}) | L^{(1)}_{m,k}-L^{(1)}_{m,\sigma^{(1)}_n}| \,\big|\,\bY^{(1)}_0=\by\big) \cr
%
%
&\le \frac{u_+}{k} \mathbb{E}\Big(\sum_{n=0}^{k} 1(\sigma^{(1)}_n\le k<\sigma^{(1)}_{n+1}) \,\Big|\,\bY^{(1)}_0=\by\Big) 
= \frac{u_+}{k}.
\label{eq:phiA2_g^{(1)}} 
\end{align}
Hence, there exists a positive integer $u_{1c}^*$ such that if $k\ge u_{1c}^*$, then for every $\by\in\calV^{(1)}_0$ and $m\in\{1,2\}$, $|\phi^c_{m,\by}(k)|<\varepsilon/4$.
%
For $\by\in\calS^{(1)}$ and for $k\ge 1$, define $q_{\by}(k)$ as 
\begin{align*}
&q_{\by}(k)= \frac{1}{k} \sum_{l=1}^{k} \sum_{\by'\in\calV^{(1)}_+} u_+ \mathbb{P}(\sigma^{(1)}_l\le k,\,\bY^{(1)}_{\sigma^{(1)}_{l-1}}=\by'\,|\,\bY^{(1)}_0=\by), 
\end{align*}
then we have 
\begin{align}
q_{\by}(k) 
&= \frac{1}{k} \sum_{l=1}^{k} \mathbb{E}\big(u^{(1)}(\bY^{(1)}_{\sigma^{(1)}_{l-1}})\,1(\sigma^{(1)}_l\le k)\,1(\bY^{(1)}_{\sigma^{(1)}_{l-1}}\in\calV^{(1)}_+)\,|\,\bY^{(1)}_0=\by\big) 
= q^a_{\by}(k) - q^b_{\by}(k), 
\label{eq:qAB_psi12}
\end{align}
where 
\begin{align*}
&q^a_{\by}(k) = \frac{1}{k} \mathbb{E}\Big( \sum_{l=1}^{k} u^{(1)}(\bY^{(1)}_{\sigma^{(1)}_{l-1}})\,1(\sigma^{(1)}_l\le k)\,|\,\bY^{(1)}_0=\by\Big), \\
&q^b_{\by}(k) = \frac{1}{k} \sum_{l=1}^{k} \mathbb{P}(\sigma^{(1)}_l\le k,\,\bY^{(1)}_{\sigma^{(1)}_{l-1}}\in\calV^{(1)}_0\,|\,\bY^{(1)}_0=\by) 
\end{align*}
and we use the fact that $u^{(1)}(\by')=u_+$ for $\by'\in\calV^{(1)}_+$.
Since $\sigma^{(1)}_l=\sum_{n=1}^{l} u^{(1)}(\bY^{(1)}_{\sigma^{(1)}_{n-1}})$, we have
\[
\sum_{l=1}^{k} u^{(1)}(\bY^{(1)}_{\sigma^{(1)}_{l-1}})\,1(\sigma^{(1)}_l\le k) = \sum_{l=1}^{k} \sigma^{(1)}_l\,1(\sigma^{(1)}_l\le k<\sigma^{(1)}_{l+1}), 
\] 
and this leads us to $(k-u_+)/k < q^a_{\by}(k) \le k/k = 1$. Hence, there exists a positive integer $u^*_{1b,a}$ such that if $k\ge u^*_{1b,a}$, then for every $\by\in\calS^{(1)}$, $1-\varepsilon/(8 \bar{a}^{(+)}_{max})< q^a_{\by}(k) \le 1$, where $\bar{a}^{(+)}_{max}=\max\{1,|\bar{a}^{(+)}_1|,|\bar{a}^{(+)}_2|\}$. 
We have 
\begin{align}
q^b_{\by}(k)
%
&\le \frac{1}{k} \sum_{l=0}^{k-1} \mathbb{P}(\bY^{(1)}_l\in\calV^{(1)}_0\,|\,\bY^{(1)}_0=\by), 
%
\end{align}
where $\mathbb{P}(\bY^{(1)}_l\in\calV^{(1)}_0\,|\,\bY^{(1)}_0=\by)=\mathbb{P}\big((L^{(1)}_{2,l},J^{(1)}_l)\in\calV_0^*\,|\,(L^{(1)}_{2,0},J^{(1)}_0)=(l_2,j)\big)$. 
Since $\calV_0^*$ is finite and every state of $\bar{\calL}^{(1)}=\{(L^{(1)}_{2,n},J^{(1)}_n)\}$ is transient or null recurrent, there exists a positive integer $u_{1b,b}^*$ such that if $k\ge u_{1b,b}^*$, then for every $\by\in\calV^{(1)}_0$, $0\le q^b_{\by}(k)<\varepsilon/(8 \bar{a}^{(+)}_{max})$. 
Hence, by expression (\ref{eq:qAB_psi12}), letting $u^*_{1b}=\max\{u^*_{1b,a},u^*_{1b,b}\}$, we see that if $k\ge u^*_{1b}$, then for every $\by\in\calV^{(1)}_0$, $1-\varepsilon/(4 \bar{a}^{(+)}_{max}) < q_{\by}(k) \le 1$.
Under the condition of the proposition, for every $\by'\in\calV^{(1)}_+$ and $m\in\{1,2\}$, $|g^{(+)}_{m,\by'}(u_+)-\bar{a}^{(+)}_m|<\varepsilon/4$, and we have, for every $\by\in\calV^{(1)}_0$ and $m\in\{1,2\}$, 
\begin{align*}
|\phi^b_{m,{\by}}(k)-\bar{a}^{(+)}_m q_{\by}(k)| < \varepsilon/4 \cdot q_{\by}(k) \le \varepsilon/4.
\end{align*}

As a result, letting $u^*_1=\max\{u_{1a}^*,u_{1b}^*,u_{1c}^*,u_++1\}$, we see from equation (\ref{eq:gA_phiA123}) that if $k\ge u_1^*$, then for every $\by\in\calV^{(1)}_0$ and for $m\in\{1,2\}$,  
\begin{align*}
|g^{(1)}_{m,\by}(k)-\bar{a}^{(+)}_m| 
\le |\phi^a_{m,\by}(k)| + |\phi^b_{m,\by}(k)-\bar{a}^{(+)}_m q_{\by}(k)| + |\bar{a}^{(+)}_m|\, |q_{\by}(k)-1| + |\phi^c_{m,\by}(k)| <\varepsilon,
\end{align*}
and this completes the proof.
\end{proof}

\end{document}